\documentclass[11pt,letterpaper]{amsart}

\usepackage{amsmath,amssymb,amsthm,mathrsfs,bbm,comment,units,enumitem,xcolor}
\usepackage[colorlinks, linktocpage]{hyperref}

\renewcommand{\leq}{\leqslant}
\renewcommand{\geq}{\geqslant}

\newcommand{\one}{\mathbf 1}

\newtheoremstyle{mythm}
{.5\baselineskip}	
{.5\baselineskip}	
{}		
{}		
{\bf}	
{.}		
{ }		
{}		

\theoremstyle{mythm}
\newtheorem{theorem}{Theorem}	
\newtheorem{lemma}[theorem]{Lemma}
\newtheorem{proposition}[theorem]{Proposition}
\newtheorem{corollary}[theorem]{Corollary}
\newtheorem{definition}[theorem]{Definition}
\newtheorem{example}[theorem]{Example}

\newtheorem{question}{Question}

\newtheoremstyle{myclaim}
{.5\baselineskip}	
{.5\baselineskip}	
{}		
{}		
{\sc}	
{. }		
{ }		
{}		

\theoremstyle{myclaim}

\newcommand{\astfill}{\noindent\xleaders\hbox{$\ast$}\hfill\kern0pt}

\newcommand{\textdefn}[1]{\textcolor{black!25!green!15!red!10!orange}{\textbf{#1}}}

\usepackage{CormorantGaramond}

\title{Selection Games with Minimal Usco Maps}
\author{Christopher Caruvana}
\address{School of Sciences, Indiana University Kokomo, Kokomo, IN, 46902, United States of America}
\email{chcaru@iu.edu}
\urladdr{https://chcaru.pages.iu.edu/}

\subjclass[2010]{91A44, 54C60, 54C35, 54D20, 54B20}
\keywords{minimal usco maps, topological selection principles, topology of uniform convergence on compacta}
\date{\today}

\begin{document}

\maketitle

\begin{abstract}
  We establish relationships between various topological selection games involving the space of minimal usco maps
  with various topologies, including the topology of pointwise convergence and the topology
  of uniform convergence on compact sets,
  and the underlying domain using full- and limited-information strategies.
  We also tie these relationships to analogous results related to spaces of continuous functions.
  The primary games we consider include Rothberger-like games, generalized point-open games,
  strong fan-tightness games, Tkachuk's closed discrete selection game,
  and Gruenhage's \(W\)-games.
\end{abstract}

\section{Introduction}

Minimal upper-semicontinuous compact-valued functions have a rich history, apparently arising
from the study of holomorphic functions and their so-called cluster sets; see \cite{ClusterSetsBook}.
The phrase \emph{minimal usco} was coined by
Christensen \cite{Christensen1982}, where a topological game similar to the Banach-Mazur game
was considered.
In this paper, using some techniques similar to those of Hol{\'{a}} and Hol{\'{y}} \cite{HolaHoly2022},
we tie connections between a space \(X\) and the space of minimal usco maps
with the topology of uniform convergence on certain kinds of subspaces of \(X\) similar in spirit to those
appearing in \cite{ClontzHolshouser,CHCompactOpen,CHContinuousFunctions}; in particular,
most of the results come in the form of selection game equivalences or dualities,
which rely on a variety of game-related results from
\cite{ClontzDualSelection,Tkachuk2018,TkachukCpBook,TkachukTwoPoint,CHCompactOpen,CHContinuousFunctions}.
We also tie some of these results to existing results relating topological properties of a space \(X\)
with the space of continuous real-valued functions on \(X\) with various topologies.

Consequences of these results include Corollary \ref{cor:ParticularMetrizability}, which captures
\cite[Cor. 4.5]{HolaHoly2022}, that states that
\(X\) is hemicompact if and only if \(MU_k(X)\), the space of minimal usco maps into \(\mathbb R\)
on \(X\) with the topology of uniform convergence on compact subsets, is metrizable.
Corollary \ref{cor:ParticularMetrizability} also shows that \(X\) is hemicompact if and only if \(MU_k(X)\)
is not discretely selective.
Corollary \ref{cor:BigRothberger} contains the assertion that \(X\) is \(k\)-Rothberger if and only if
\(MU_k(X)\) has strong countable fan-tightness at \(\mathbf 0\), the constant \(\{0\}\) function.
We end with Corollary \ref{cor:FinalCorollary}, which characterizes a space having a countable
so-called weak \(k\)-covering number in terms of strategic information in selection games.

\section{Preliminaries}

We use the word \emph{space} to mean \emph{topological space}.
When the parent space is understood from context, we use the notation
\(\mathrm{int}(A)\), \(\mathrm{cl}(A)\), and \(\partial A\) to denote
the interior, closure, and boundary of \(A\), respectively.
If we must specify the topological space \(X\), we use \(\mathrm{int}_X(A)\),
\(\mathrm{cl}_X(A)\), and \(\partial_X A\).

Given a function \(f : X \to Y\), we denote the graph of \(f\) by
\(\mathrm{gr}(f) = \{ \langle x , f(x) \rangle : x \in X \}\).
For a set \(X\), we let \(\wp(X)\) denote the set of subsets of \(X\) and
\(\wp^+(X) = \wp(X) \setminus \{\emptyset\}\).
For sets \(X\) and \(Y\), we let \[\mathrm{Fn}(X,Y) = \bigcup_{A \in \wp^+(X)} Y^A;\]
that is, \(\mathrm{Fn}(X,Y)\) is the collection of all \(Y\)-valued functions
defined on non-empty subsets of \(X\).

When a set \(X\) is implicitly serving as the parent space in context,
given \(A \subseteq X\), we will let \(\one_A\) be the indicator function
for \(A\); that is, \(\one_A : X \to \{0,1\}\) is defined by the rule
\[\one_A(x) = \begin{cases} 1, & x \in A\\ 0, & x \not\in A \end{cases}.\]

For any set \(X\), we let \(X^{<\omega}\) denote the set of finite sequences of \(X\),
\([X]^{<\omega}\) denote the set of finite subsets of \(X\), and, for any cardinal \(\kappa\),
\([X]^{\kappa}\) denote the set of \(\kappa\)-sized subsets of \(X\).

For a space \(X\), we let \(K(X)\) denote the set of all non-empty compact subsets of \(X\).
We let \(\mathbb K(X)\) denote the set \(K(X)\) endowed with the Vietoris topology;
that is, the topology with basis consisting of sets of the form
\[[U_1,U_2,\ldots, U_n]
= \left\{ K \in \mathbb K(X) : K \subseteq \bigcup_{j=1}^n U_j \wedge K^n \cap \prod_{j=1}^n U_j \neq \emptyset \right\}.\]
For more about this topology, see \cite{MichaelSubsets}.

\begin{definition}
  For a set \(X\), we say that a family \(\mathcal A \subseteq \wp^+(X)\) is an \textdefn{ideal of sets} if
  \begin{itemize}
      \item
      for \(A , B \in \mathcal A\), \(A \cup B \in\mathcal A\), and
      \item
      for every \(x \in X\), \(\{x\} \in \mathcal A\).
  \end{itemize}
  If \(X\) is a space, we say that an ideal of sets \(\mathcal A\) is an \textdefn{ideal of closed sets}
  if \(\mathcal A\) consists of closed sets
\end{definition}
Throughout, we will assume that any ideal of closed sets under consideration doesn't contain
the entire space \(X\).
Two ideals of closed sets of primary interest are
\begin{itemize}
    \item
    the collection of non-empty finite subsets of an infinite space \(X\) and
    \item
    the collection of non-empty compact subsets of a non-compact space \(X\).
\end{itemize}

\subsection{Selection Games}
General topological games have a long history, a lot of which can be gathered from Telg{\'a}rsky's survey
\cite{TelgarskySurvey}.
In this paper, we will be dealing only with single-selection games of countable length.

\begin{definition}
  Given sets \(\mathcal A\) and \(\mathcal B\), we define the \textdefn{single-selection game}
  \(\mathsf G_1(\mathcal A, \mathcal B)\) as follows.
  \begin{itemize}
      \item
      For each \(n\in\omega\), One chooses \(A_n \in \mathcal A\) and Two responds with \(x_n \in A_n\).
      \item
      Two is declared the winner if \(\{ x_n : n \in \omega \} \in \mathcal B\).
      Otherwise, One wins.
  \end{itemize}
\end{definition}

\begin{definition}
    We define strategies of various strength below.
    \begin{itemize}
        \item
        We use two forms of \emph{full-information} strategies.
        \begin{itemize}
            \item
            A \textdefn{strategy for player One} in \(\textsf{G}_1(\mathcal A, \mathcal B)\) is a function
            \(\sigma:(\bigcup \mathcal A)^{<\omega} \to \mathcal A\).
            A strategy \(\sigma\) for One is called \textdefn{winning} if whenever
            \(x_n \in \sigma\langle x_k : k < n \rangle\) for all \(n \in \omega\), \(\{x_n: n\in\omega\} \not\in \mathcal B\).
            If player One has a winning strategy, we write \(\mathrm{I} \uparrow \textsf{G}_1(\mathcal A, \mathcal B)\).
            \item
            A \textdefn{strategy for player Two} in \(\textsf{G}_1(\mathcal A, \mathcal B)\) is a function
            \(\tau:\mathcal A^{<\omega} \to \bigcup \mathcal A\).
            A strategy \(\tau\) for Two is \textdefn{winning} if whenever \(A_n \in \mathcal A\) for all \(n \in \omega\),
            \(\{\tau(A_0,\ldots,A_n) : n \in \omega\} \in \mathcal B\).
            If player Two has a winning strategy, we write \(\mathrm{II} \uparrow \textsf{G}_1(\mathcal A, \mathcal B)\).
        \end{itemize}
        \item
        We use two forms of \emph{limited-information} strategies.
        \begin{itemize}
            \item
            A \textdefn{predetermined strategy} for One is a strategy which only considers the current turn number.
            We call this kind of strategy predetermined because One is not reacting to Two's moves.
            Formally it is a function \(\sigma: \omega \to \mathcal A\).
            If One has a winning predetermined strategy, we write \(\mathrm{I} \underset{\mathrm{pre}}{\uparrow} \textsf{G}_1(\mathcal A, \mathcal B)\).
            \item
            A \textdefn{Markov strategy} for Two is a strategy which only considers the most recent move of player One and the current turn number.
            Formally it is a function \(\tau:\mathcal A \times \omega \to \bigcup \mathcal A\).
            If Two has a winning Markov strategy, we write \(\mathrm{II} \underset{\mathrm{mark}}{\uparrow} \textsf{G}_1(\mathcal A, \mathcal B)\).
        \end{itemize}
    \end{itemize}
\end{definition}

The reader may be more familiar with selection principles than selection games.
For more details on selection principles and relevant references, see \cite{KocinacSelectedResults,ScheepersI}.
\begin{definition}
    Let \(\mathcal A\) and \(\mathcal B\) be collections.
    The \textdefn{single-selection principle} \(\textsf{S}_1(\mathcal A, \mathcal B)\) for a space \(X\) is the following property.
    Given any \(A \in \mathcal A^\omega\), there exists \(\vec x \in \prod_{n\in\omega} A_n\) so that
    \(\{ \vec{x}_n : n \in \omega \} \in \mathcal B\).
\end{definition}
As mentioned in \cite[Prop. 15]{ClontzDualSelection}, \(\textsf{S}_1(\mathcal A, \mathcal B)\) holds if and only if \(\mathrm{I} \underset{\mathrm{pre}}{\not\uparrow} \textsf{G}_1(\mathcal{A},\mathcal{B})\).
Hence, we may establish equivalences between certain selection principles by addressing the corresponding
selection games.

\begin{definition}
  For a space \(X\), an open cover \(\mathscr U\) of \(X\) is said to be \textdefn{non-trivial}
  if \(\emptyset \not\in \mathscr U\) and \(X \not\in \mathscr U\).
\end{definition}
\begin{definition}
  Let \(X\) be a space and \(\mathcal A\) be a set of closed subsets of \(X\).
  We say that a non-trivial cover \(\mathscr U\) of \(X\) is an \textdefn{\(\mathcal A\)-cover}
  if, for every \(A \in \mathcal A\), there exists \(U \in \mathscr U\) so that
  \(A \subseteq U\).
\end{definition}

\begin{definition}
    For a collection \(\mathcal A\), we let \(\neg \mathcal A\) denote the collection of
    sets which are not in \(\mathcal A\).
    We also define the following classes for a space \(X\) and a collection \(\mathcal A\) of closed subsets of \(X\).
    \begin{itemize}
    \item
    \(\mathscr T_X\) is the family of all proper non-empty open subsets of \(X\).
    \item
    For \(x \in X\), \(\mathscr N_{X,x} = \{ U \in \mathscr T_X : x \in U \}\).
    \item
    For \(A \in \wp^+(X)\), \(\mathscr N_X(A) = \{ U \in \mathscr T_X : A \subseteq U \}\).
    \item
    \(\mathscr N_X[\mathcal A] = \{ \mathscr N_X(A) : A \in \mathcal A \}\),
    \item
    \(\mathrm{CD}_X\) is the set of all closed discrete subsets of \(X\).
    \item
    \(\mathscr D_X\) is the set of all dense subsets of \(X\).
    \item
    For \(x \in X\), \(\Omega_{X,x} = \{A \subseteq X : x \in \mathrm{cl}(A)\}\).
    \item
    For \(x\in X\), \(\Gamma_{X,x}\) is the set of all sequences of \(X\) converging to \(x\).
    \item
    \(\mathcal O_X\) is the set of all non-trivial open covers of \(X\).
    \item
    \(\mathcal O_X(\mathcal A)\) is the set of all \(\mathcal A\)-covers.
    \item
    \(\Lambda_X(\mathcal A)\) is the set of all \(\mathcal A\)-covers \(\mathscr U\) with the property that,
    for every \(A \in \mathcal A\), \(\{ U \in \mathscr U : A \subseteq U \}\) is infinite.
    \item
    \(\Gamma_X(\mathcal A)\) is the set of all countable \(\mathcal A\)-covers \(\mathscr U\) with the property that,
    for every \(A \in \mathcal A\), \(\{ U \in \mathscr U : A \subseteq U \}\) is co-finite.
\end{itemize}
\end{definition}

Note that, in our notation, \(\mathcal O_X([X]^{<\omega})\) is the set of all \(\omega\)-covers of \(X\),
which we will denote by \(\Omega_X\), and that \(\mathcal O_X(K(X))\) is the set of all \(k\)-covers of \(X\),
which we will denote by \(\mathcal K_X\).
We also use \(\Gamma_\omega(X)\) to denote \(\Gamma_X([X]^{<\omega})\) and \(\Gamma_k(X)\) to denote
\(\Gamma_X(K(X))\).

Also note that \(\textsf{S}_{1}(\mathcal O_X, \mathcal O_X)\) is the Rothberger property
and \(\mathsf G_1(\mathcal O_X, \mathcal O_X)\) is the Rothberger game.
If we let \(\mathbb P_X = \{ \mathscr N_{X,x} : x \in X \}\), then
\(\mathsf G_1(\mathbb P_X, \neg \mathcal O)\) is isomorphic to the point-open game studied by Galvin \cite{Galvin1978}
and Telg{\'{a}}rsky \cite{Telgarsky1975}.
The games \(\mathsf G_1(\mathscr N_{X,x} , \neg \Gamma_{X,x})\) and \(\mathsf G_1(\mathscr N_{X,x} , \neg \Omega_{X,x})\)
are two variants of Gruenhage's \(W\)-game (see \cite{Gruenhage1976}).
We refer to \(\mathsf G_1(\mathscr N_{X,x} , \neg \Gamma_{X,x})\) as Gruenhage's converging \(W\)-game
and \(\mathsf G_1(\mathscr N_{X,x} , \neg \Omega_{X,x})\) as Gruenhage's clustering \(W\)-game.
The games \(\mathsf G_1(\mathscr T_X, \neg \Omega_{X,x})\) and \(\mathsf G_1(\mathscr T_X, \mathrm{CD}_X)\) were introduced
by Tkachuk (see \cite{Tkachuk2018,TkachukTwoPoint}) and tied to Gruenhage's \(W\)-games in \cite{TkachukTwoPoint,ClontzHolshouser}.
The strong countable dense fan-tightness game at \(x\) is \(\mathsf G_1(\mathscr D_X , \Omega_{X,x} )\)
and the strong countable fan-tightness game at \(x\) is \(\mathsf G_1(\Omega_{X,x} , \Omega_{X,x} )\) (see \cite{BarmanDown2011}).

\begin{lemma}[See {\cite[Lemma 4]{CHContinuousFunctions}}] \label{lem:IdealLargeCovers}
    For a space \(X\) and an ideal of closed sets \(\mathcal A\) of \(X\),
    \(\mathcal O_X(\mathcal A) = \Lambda_X(\mathcal A)\).
\end{lemma}

In what follows, we say that \(\mathcal G\) is a \emph{selection game} if there exist classes \(\mathcal A\), \(\mathcal B\)
so that \(\mathcal G = \textsf{G}_1(\mathcal A, \mathcal B)\).

Since we work with full- and limited-information strategies, we reflect this in our definitions of game equivalence and duality.
\begin{definition}
  We say that two selection games \(\mathcal G\) and \(\mathcal H\) are \textdefn{equivalent},
  denoted \(\mathcal G \equiv \mathcal H\), if the following hold:
  \begin{itemize}
      \item
      \(\mathrm{II} \underset{\mathrm{mark}}{\uparrow} \mathcal G \iff \mathrm{II} \underset{\mathrm{mark}}{\uparrow} \mathcal H\)
      \item
      \(\mathrm{II} \uparrow \mathcal G \iff \mathrm{II} \uparrow \mathcal H\)
      \item
      \(\mathrm{I} \not\uparrow \mathcal G \iff \mathrm{I} \not\uparrow \mathcal H\)
      \item
      \(\mathrm{I} \underset{\mathrm{pre}}{\not\uparrow} \mathcal G \iff \mathrm{I} \underset{\mathrm{pre}}{\not\uparrow} \mathcal H\)
  \end{itemize}
\end{definition}
We also use a preorder on selection games.
\begin{definition}
  Given selection games \(\mathcal G\) and \(\mathcal H\), we say that \(\mathcal G \leq_{\mathrm{II}} \mathcal H\) if the following implications hold:
  \begin{itemize}
    \item
    \(\mathrm{II} \underset{\mathrm{mark}}{\uparrow} \mathcal G \implies \mathrm{II} \underset{\mathrm{mark}}{\uparrow} \mathcal H\)
    \item
    \(\mathrm{II} \uparrow \mathcal G \implies \mathrm{II} \uparrow \mathcal H\)
    \item
    \(\mathrm{I} \not\uparrow \mathcal G \implies \mathrm{I} \not\uparrow \mathcal H\)
    \item
    \(\mathrm{I} \underset{\mathrm{pre}}{\not\uparrow} \mathcal G \implies \mathrm{I} \underset{\mathrm{pre}}{\not\uparrow} \mathcal H\)
  \end{itemize}
\end{definition}
Note that \(\leq_{\mathrm{II}}\) is transitive and that if \(\mathcal G \leq_{\mathrm{II}} \mathcal H\) and \(\mathcal H \leq_{\mathrm{II}} \mathcal G\), then \(\mathcal G \equiv \mathcal H\).
We use the subscript of \(\mathrm{II}\) since each implication in the definition of \(\leq_{\mathrm{II}}\)
is related to a transference of winning plays by Two.
\begin{definition}
  We say that two selection games \(\mathcal G\) and \(\mathcal H\) are \textdefn{dual} if the following hold:
  \begin{itemize}
    \item
    \(\mathrm{I} \uparrow \mathcal G \iff \mathrm{II} \uparrow \mathcal H\)
    \item
    \(\mathrm{II} \uparrow \mathcal G \iff \mathrm{I} \uparrow \mathcal H\)
    \item
    \(\mathrm{I} \underset{\mathrm{pre}}{\uparrow} \mathcal G \iff \mathrm{II} \underset{\mathrm{mark}}{\uparrow} \mathcal H\)
    \item
    \(\mathrm{II} \underset{\mathrm{mark}}{\uparrow} \mathcal G \iff \mathrm{I} \underset{\mathrm{pre}}{\uparrow} \mathcal H\)
  \end{itemize}
\end{definition}
We note one important way in which equivalence and duality interact.
\begin{lemma} \label{lem:BasicEquivDual}
    Suppose \(\mathcal G_1, \mathcal G_2, \mathcal H_1\), and \(\mathcal H_2\) are selection games
    so that \(\mathcal G_1\) is dual to \(\mathcal H_1\)
    and \(\mathcal G_2\) is dual to \(\mathcal H_2\).
    Then, if \(\mathcal G_1 \leq_{\mathrm{II}} \mathcal G_2\), \(\mathcal H_2 \leq_{\mathrm{II}} \mathcal H_1\).
    Consequently, if \(\mathcal G_1 \equiv \mathcal G_2\), then \(\mathcal H_1 \equiv \mathcal H_2\).
\end{lemma}
We will use consequences of \cite[Cor. 26]{ClontzDualSelection} to see that a few classes of selection games are dual.
\begin{lemma} \label{lem:DualGames}
    Let \(\mathcal A\) be an ideal of closed sets of a space \(X\) and \(\mathcal B\) be a collection.
    \begin{enumerate}[label=(\roman*)]
        \item
        By \cite[Cor. 3.4]{CHCompactOpen} and \cite[Thm. 38]{ClontzDualSelection},
        \(\mathsf G_1(\mathcal O_X(\mathcal A), \mathcal B)\) and \(\mathsf G_1(\mathscr N_X[\mathcal A], \neg \mathcal B)\)
        are dual. (Note that this is a general form of the duality of the Rothberger game and the point-open game.)
        \item
        By \cite[Cor. 33]{ClontzDualSelection},
        \(\mathsf G_1(\mathscr D_X,\mathcal B)\) and \(\mathsf G_1(\mathscr T_X, \neg\mathcal B)\) are dual.
        \item
        By \cite[Cor. 35]{ClontzDualSelection}, for \(x \in X\), \(\mathsf G_1(\Omega_{X,x} , \mathcal B)\) and
        \(\mathsf G_1(\mathscr N_{X,x} , \neg \mathcal B)\)
        are dual.
    \end{enumerate}
\end{lemma}

We now state the translation theorems we will be using to establish some game equivalences.
\begin{theorem}[{\cite[Thm. 12]{CHContinuousFunctions}}]\label{TranslationTheorem}
    Let \(\mathcal A\), \(\mathcal B\), \(\mathcal C\), and \(\mathcal D\) be collections.
    Suppose there are functions \(\overleftarrow{T}_{\mathrm{I},n}:\mathcal B \to \mathcal A\) and
    \(\overrightarrow{T}_{\mathrm{II},n}: \left( \bigcup \mathcal A \right) \times \mathcal B \to \bigcup \mathcal B\)
    for each \(n \in \omega\), so that
    \begin{enumerate}[label=(\roman*)]
        \item
        if \(x \in \overleftarrow{T}_{\mathrm{I},n}(B)\), then \(\overrightarrow{T}_{\mathrm{II},n}(x,B) \in B\), and
        \item
        if \(\langle x_n : n \in \omega \rangle \in \prod_{n\in\omega} \overleftarrow{T}_{\mathrm{I},n}(B_n)\)
        and \(\{x_n : n\in\omega\} \in \mathcal C\),
        then \(\left\{\overrightarrow{T}_{\mathrm{II},n}(x_n,B_n) : n \in \omega \right\} \in \mathcal D\).
    \end{enumerate}
    Then \(\mathsf G_1(\mathcal A,\mathcal C) \leq_{\mathrm{II}} \mathsf G_1(\mathcal B, \mathcal D)\).
\end{theorem}

Similar results to Theorem \ref{TranslationTheorem} for longer length games and finite-selection games,
even with simplified hypotheses, can be found in \cite{CHContinuousFunctions,CHHyperspaces,CHVietoris}.

We will also need a separation axiom for some results in the sequel.
\begin{definition}
  Let \(X\) be a space and \(\mathcal A\) be an ideal of closed subsets of \(X\).
  We say that \(X\) is \textdefn{\(\mathcal A\)-normal} if, given any \(A \in \mathcal A\)
  and \(U \subseteq X\) open with \(A \subseteq U\), there exists an open set \(V\)
  so that \(A \subseteq V \subseteq \mathrm{cl}(V) \subseteq U\).

  We also say that \(X\) is \textdefn{functionally \(\mathcal A\)-normal} if, for \(A \in \mathcal A\)
  and \(U \subseteq X\) open with \(A \subseteq U\), there exists a continuous function \(f : X \to \mathbb R\)
  so that \(f[A] = \{0\}\) and \(f[X \setminus U] = \{1\}\).
\end{definition}
Note that, if \(X\) is \(\mathcal A\)-normal, then \(X\) is regular.
If \(\mathcal A \subseteq K(X)\) and \(X\) is regular, then \(X\) is \(\mathcal A\)-normal.
If \(X\) is Tychonoff and \(\mathcal A \subseteq K(X)\), then \(X\) is functionally
\(\mathcal A\)-normal.

\subsection{Uniform Spaces} \label{section:UniformSpaces}

To introduce the basics of uniform spaces we'll need in this paper,
we mostly follow \cite[Chapter 6]{Kelley}.

We recall the standard notation involved with uniformities.
Let \(X\) be a set.
The diagonal of \(X\) is \(\Delta_X = \{ \langle x , x \rangle : x \in X \}\).
For \(E \subseteq X^2\), \(E^{-1} = \{ \langle y,x \rangle : \langle x,y \rangle \in E \}\).
If \(E = E^{-1}\), then \(E\) is said to be \emph{symmetric}.
If \(E , F \subseteq X^2\),
\[E \circ F = \{\langle x, z \rangle : (\exists y \in X)\ \langle x,y \rangle \in F \wedge \langle y , z \rangle \in E \}.\]
For \(E \subseteq X^2\), we let \(E[x] = \{ y \in X : \langle x , y \rangle \in E \}\) and
\(E[A] = \bigcup_{x \in A} E[x]\).
\begin{definition}
  A \textdefn{uniformity} on a set \(X\) is a set \(\mathcal E \subseteq \wp^+(X^2)\) which
  satisfies the following properties:
  \begin{itemize}
      \item
      For every \(E \in \mathcal E\), \(\Delta_X \subseteq E\).
      \item
      For every \(E\in \mathcal E\), \(E^{-1} \in \mathcal E\).
      \item
      For every \(E \in \mathcal E\), there exists \(F \in \mathcal E\) so that \(F \circ F \subseteq E\).
      \item
      For \(E , F \in \mathcal E\), \(E \cap F \in \mathcal E\).
      \item
      For \(E \in \mathcal E\) and \(F \subseteq X^2\), if \(E \subseteq F\), \(F \in \mathcal E\).
  \end{itemize}
  If, in addition, \(\Delta_X = \bigcap \mathcal E\), we say that the uniformity
  \(\mathcal E\) is \textdefn{Hausdorff}.
  By an \textdefn{entourage} of \(X\), we mean \(E \in \mathcal E\).
  The pair \((X, \mathcal E)\) is called a \textdefn{uniform space}.
\end{definition}
\begin{definition}
  For a set \(X\), we say that \(\mathcal B \subseteq \wp^+(X^2)\) is
  a \textdefn{base for a uniformity} if
  \begin{itemize}
      \item
      for every \(B \in \mathcal B\), \(\Delta_X \subseteq B\);
      \item
      for every \(B \in \mathcal B\), there is some \(A \in \mathcal B\) so that \(A \subseteq B^{-1}\);
      \item
      for every \(B \in \mathcal B\), there is some \(A \in \mathcal B\) so that \(A \circ A \subseteq B\); and
      \item
      for \(A , B \in \mathcal B\), there is some \(C \in \mathcal B\) so that \(C \subseteq A \cap B\).
  \end{itemize}
  If the uniformity generated by \(\mathcal B\) is \(\mathcal E\),
  we say that \(\mathcal B\) is a \textdefn{base} for \(\mathcal E\).
\end{definition}

If \((X,\mathcal E)\) is a uniform space, then the uniformity \(\mathcal E\)
generates a topology on \(X\) in the following way: \(U \subseteq X\) is declared
to be open provided that, for every \(x \in U\), there is some \(E \in \mathcal E\)
so that \(E[x] \subseteq U\).
An important result about this topology is
\begin{theorem}[see \cite{Kelley}] \label{thm:UniformMetrizability}
    A Hausdorff uniform space \((X,\mathcal E)\) is metrizable if and only if
    \(\mathcal E\) has a countable base.
\end{theorem}
With the topology induced by the uniformity, we endow \(X^2\) with the
resulting product topology.
\begin{lemma}[See \cite{Kelley}] \label{lem:ClosedEntourages}
    The family of open (or closed) symmetric entourages of a uniformity \(\mathcal E\)
    is a base for \(\mathcal E\).
\end{lemma}

For a uniform space \((X,\mathcal E)\), there is a natural way to define a uniformity on \(K(X)\)
which is directly analogous to the Pompeiu-Hausdorff distance defined in the context of metric spaces.
\begin{definition}
    Let \((X,\mathcal E)\) be a uniform space and, for \(E \in \mathcal E\), define
    \[hE = \{\langle K , L \rangle \in K(X)^2 : K \subseteq E[L] \wedge L \subseteq E[K]\}.\]
\end{definition}

Just as the Pompeiu-Hausdorff distance on compact subsets generates the Vietoris topology,
the analogous uniformity also generates the Vietoris topology.
\begin{theorem}[see {\cite[Chapter 2]{CastaingValadier1977}}] \label{thm:UniformVietoris}
    For a uniform space \((X,\mathcal E)\), \(\mathcal B = \{ hE : E \in \mathcal E \}\)
    is a base for a uniformity on \(K(X)\);
    the topology generated by the uniform base \(\mathcal B\) is the Vietoris topology.
\end{theorem}

For the set of functions from a space \(X\) to a uniform space \((Y,\mathcal E)\),
we review the uniformity which generates the topology of uniform convergence on a family
of subsets of \(X\).
For this review, we mostly follow \cite[Chapter 7]{Kelley}.
\begin{definition} \label{def:UniformStructure}
    For the set \(Y^X\) of functions from a set \(X\) to a uniform space \((Y,\mathcal E)\), we define,
    for \(A \in \wp^+(X)\) and \(E \in \mathcal E\),
    \[
        \mathbf U(A,E) = \{\langle f , g \rangle \in (Y^X)^2 : (\forall x \in A)\ \langle f(x), g(x) \rangle \in E\}.
    \]
    For the set of functions \(X \to \mathbb K(Y)\), we let \(\mathbf W(A,E) = \mathbf U(A,hE)\).
\end{definition}
If \(\mathcal B\) is a base for a uniformity on \(Y\) and \(\mathcal A\) is an ideal of subsets of \(X\),
then \(\{ \mathbf U(A,B) : A \in \mathcal A , B \in \mathcal B \}\) forms a base for a uniformity on \(Y^X\).
The corresponding topology generated by this base for a uniformity is the topology of uniform
convergence on \(\mathcal A\).
Consequently, \(\{ \mathbf W(A,B) : A \in \mathcal A , B \in \mathcal B \}\) is a base for a uniformity
on \(\mathbb K(Y)^X\).

\subsection{Usco Mappings}

In this section, we introduce the basic facts of usco mappings needed for this paper.
Of primary use is Theorem \ref{thm:HolaHolyChar} which offers a convenient characterization
of minimal usco maps.

A \emph{set-valued} function from \(X\) to \(Y\) is a function \(\Phi : X \to \wp(Y)\).
These are sometimes also referred to as \emph{multi-functions}.
\begin{definition}
  A set-valued function \(\Phi : X \to \wp(Y)\) is said to be \textdefn{upper semicontinuous} if, for every open
  \(V \subseteq Y\), \[\Phi^\leftarrow(V) := \{ x \in X : \Phi(x) \subseteq V\}\] is open in \(X\).
  An \textdefn{usco} map from a space \(X\) to \(Y\) is a set-valued map \(\Phi\) from \(X\) to \(Y\) which is upper semicontinuous
  and whose range contained in \(\mathbb K(Y)\).
  An usco map \(\Phi : X \to \mathbb K(Y)\) is said to be \textdefn{minimal} if its graph minimal with respect to the
  \(\subseteq\) relation.
  Let \(MU(X,Y)\) denote the collection of all minimal usco maps \(X \to \mathbb K(Y)\).
\end{definition}

It is clear that any map \(\Phi : X \to K(Y)\) is usco if and only if \(\Phi\) is continuous
relative to the \emph{upper Vietoris topology} on \(K(X)\), which is the topology generated
by the sets \(\{ K \in K(X) : K \subseteq U \}\) for open \(U \subseteq X\).
However, inspired by Theorem \ref{thm:UniformVietoris}, we maintain that the full Vietoris
topology is the desirable topology.
At minimum, for most spaces of interest, the full Vietoris topology is Hausdorff.
Moreover, the full Vietoris topology on the set of compact subsets has other desirable properties
that are intimately related with \(X\), like being metrizable when, and only when, \(X\) is metrizable.

As above, it is also clear that any continuous \(\Phi : X \to \mathbb K(Y)\) is usco and
that there are continuous \(\Phi : X \to \mathbb K(Y)\) which are not minimal.
As Example \ref{ex:UscoNotCont} will demonstrate, there are minimal usco
maps which are not continuous.

\begin{definition}
  Suppose \(\Phi : X \to \wp^+(Y)\).
  We say that a function \(f : X \to Y\) is a \textdefn{selection} of \(\Phi\)
  if \(f(x) \in \Phi(x)\) for every \(x \in X\).
  We let \(\mathrm{sel}(\Phi)\) be the set of all selections of \(\Phi\).

  If \(D \subseteq X\) is dense and \(f : D \to Y\) is so that \(f(x) \in \Phi(x)\)
  for each \(x \in D\), we say that \(f\) is a \textdefn{densely defined selection}
  of \(\Phi\).
\end{definition}

Recall that a point \(x\) is said to be an \emph{accumulation point}
of a net \(\langle x_\lambda : \lambda \in \Lambda \rangle\) if,
for every open neighborhood \(U\) of \(x\) and every \(\lambda \in \Lambda\),
there exists \(\mu \geq \lambda\) so that \(x_\mu \in U\).

The notion of subcontinuity was introduced by Fuller \cite{Fuller1968} which can be extended
to so-called densely defined functions in the following way.
See also \cite{Lechicki1990}.
\begin{definition}
    Suppose \(D \subseteq X\) is dense.
    We say that a function \(f : D \to Y\) is \textdefn{subcontinuous} if, for every \(x \in X\)
    and every net \(\langle x_\lambda : \lambda \in \Lambda \rangle\) in \(D\) with \(x_\lambda \to x\),
    \(\langle f(x_\lambda) : \lambda \in \Lambda \rangle\) has an accumulation point.
\end{definition}

The notion of semi-open sets was introduced by Levine \cite{Levine1963}.
\begin{definition}
  For a space \(X\), a set \(A \subseteq X\) is said to be \textdefn{semi-open} if
  \(A \subseteq \mathrm{cl}\ \mathrm{int}(A)\).
\end{definition}

The notion of quasicontinuity was introduced by Kempisty \cite{Kempisty1932}
and surveyed by Neubrunn \cite{Neubrunn}.
\begin{definition}
  A function \(f : X \to Y\) is said to be \textdefn{quasicontinuous} if, for each open
  \(V \subseteq Y\), \(f^{-1}(V)\) is semi-open in \(X\).

  If \(D \subseteq X\) is dense and \(f : D \to Y\), we will say that \(f\) is quasicontinuous
  if it is quasicontinuous on \(D\) with the subspace topology.
\end{definition}

\begin{definition}
  For \(f \in \mathrm{Fn}(X,Y)\), define \(\overline{f} : X \to \wp(Y)\) by the rule
  \[\overline{f}(x) = \{ y \in Y : \langle x , y \rangle \in \mathrm{cl}\ \mathrm{gr}(f) \}.\]
\end{definition}
\begin{theorem}[Hol{\'{a}}, Hol{\'{y}} \cite{HolaHoly2014}] \label{thm:HolaHolyChar}
  Suppose \(Y\) is regular and that \(\Phi : X \to \wp^+(X)\).
  Then the following are equivalent:
  \begin{enumerate}[label=(\roman*)]
    \item \label{tfae:MinimalUsco}
    \(\Phi\) is minimal usco.
    \item \label{tfae:AllSubQuasi}
    Every selection \(f\) of \(\Phi\) is subcontinuous, quasicontinuous, and \(\Phi = \overline{f}\).
    \item \label{tfae:SomeSubQuasi}
    There exists a selection \(f\) of \(\Phi\) which is subcontinuous,
    quasicontinuous, and \(\Phi = \overline{f}\).
    \item \label{tfae:SomeDenseSubQuasi}
    There exists a densely defined selection \(f\) of \(\Phi\) which is subcontinuous,
    quasicontinuous, and \(\Phi = \overline{f}\).
  \end{enumerate}
\end{theorem}
A consequence of Theorem \ref{thm:HolaHolyChar} we will use is that, for \(\Phi \in MU(X,\mathbb R)\),
the function \(f : X \to \mathbb R\) defined by \(f(x) = \max \Phi(x)\) is a selection
of \(\Phi\), and hence subcontinuous and quasicontinuous.

We will be using the following to construct certain functions.
\begin{lemma} \label{lem:UsefulFunction}
    Let \(f , g : X \to Y\) and \(U \in \mathscr T_X\) and define \(h : X \to Y\) by the rule
    \[h(x) = \begin{cases} f(x), & x \in \mathrm{cl}(U); \\ g(x), & x \not\in \mathrm{cl}(U).\end{cases}\]
    \begin{enumerate}[label=(\roman*)]
        \item \label{charSub}
        If \(f\) and \(g\) are subcontinuous, then \(h\) is subcontinuous.
        \item \label{charQuasi}
        If \(f\) is constant, and \(g\) is quasicontinuous,
        then \(h\) is quasicontinuous.
    \end{enumerate}
    Consequently, if \(f\) is constant and \(g\) is both subcontinuous and quasicontinuous, then
    \(h\) is both subcontinuous and quasicontinuous, which implies that \(\overline{h}\)
    is minimal usco.
\end{lemma}
\begin{proof}
    \ref{charSub}
    Suppose \(\langle x_\lambda : \lambda \in \Lambda \rangle\) is so that \(x_\lambda \to x\).
    If there is a cofinal subnet of \(\langle x_\lambda : \lambda \in \Lambda \rangle\)
    which is contained in \(X \setminus \mathrm{cl}(U)\), then we can appeal to the subcontinuity of \(g\) to see that
    \(\langle h(x_\lambda) : \lambda \in \Lambda \rangle\) has an accumulation point.
    Otherwise, we can let \(\lambda_0 \in \Lambda\) be so that, for all
    \(\lambda \geq \lambda_0\), \(x_\lambda \in \mathrm{cl}(U)\).
    Then \(\langle x_\lambda : \lambda \geq \lambda_0 \rangle\) is a net contained
    in \(\mathrm{cl}(U)\), so we can use the subcontinuity of \(f\) to establish that
    \(\langle h(x_\lambda) : \lambda \in \Lambda \rangle\) has an accumulation point.

    \ref{charQuasi}
    Suppose \(V \subseteq Y\) is open so that \(h^{-1}(V) \neq \emptyset\).
    Let \(x \in h^{-1}(V)\) and suppose \(W\) is open with \(x \in W\).
    We proceed by cases.

    If \(x \in \mathrm{cl}(U)\), then \(W \cap U \neq \emptyset\).
    Notice also that \(U \subseteq h^{-1}(V)\) in this case since \(f\) is constant.
    So \(W \cap \mathrm{int}(h^{-1}(V)) \neq \emptyset\).

    Now, suppose \(x \not\in \mathrm{cl}(U)\) and observe that \(x \in W \setminus \mathrm{cl}(U)\).
    Since \(h(x) \in V\) and \(x \not\in \mathrm{cl}(U)\), \(h(x) = g(x) \in V\).
    By the quasicontinuity of \(g\), it must be the case that
    \[(W \setminus \mathrm{cl}(U)) \cap \mathrm{int}(g^{-1}(V)) \neq \emptyset.\]
    Observe that \(\mathrm{int}(g^{-1}(V)) \setminus \mathrm{cl}(U) \subseteq h^{-1}(V)\),
    so \(W \cap \mathrm{int}(h^{-1}(V)) \neq \emptyset\).

    It follows that, in either case, \(h^{-1}(V)\) is semi-open, establishing the quasicontinuity
    of \(h\).

    For the remainder of the proof, combine \ref{charSub} and \ref{charQuasi} with
    Theorem \ref{thm:HolaHolyChar}.
\end{proof}
Note that, if \(F \subseteq X\) is closed and nowhere dense, \(\one_F\) is
not quasicontinuous since \(F\) is not semi-open.
Thus, the requirement that we use the closure of an open set in Lemma \ref{lem:UsefulFunction}\ref{charQuasi}
is, in general, necessary.
As a similar example, \(\one_{(0,1)\cup(1,2)}\) is not quasicontinuous at \(1\).

Moreover, consider \(U \in \mathscr T_X\) with \(\partial U \neq \emptyset\) and \(\mathrm{cl}(U) \neq X\).
Then \(f : X \to \mathbb R\) defined by
\[
    f(x) = \begin{cases}
    \one_{X \setminus U}(x), & x \in \mathrm{cl}(U);\\
    \one_{U}(x), & x \not\in \mathrm{cl}(U).
\end{cases}\]
is \(\mathbf 1_{\partial U}\).
So the requirement that \(f\) be constant (or something stronger than quasicontinuity, at least)
in Lemma \ref{lem:UsefulFunction}\ref{charQuasi}
is, in general, necessary.
\begin{example} \label{ex:UscoNotCont}
    Consider \(MU(\mathbb R,\mathbb R)\).
    By Lemma \ref{lem:UsefulFunction}, \(\Phi := \overline{\one}_{[0,1]}\) is minimal usco.
    However, \(\Phi\) is not continuous since
    \[\{0,1\} = \{ x\in \mathbb R : \Phi(x) \in [(-0.5,0.5),(0.5,1.5)] \}.\]
\end{example}
Hence, when \(Y\) is metrizable, studying the space \(MU(X,Y)\) is, in general, different than studying
the space of continuous functions into a metrizable space.

We will also be using the following corollary often.
\begin{corollary} \label{cor:PreUscoAgreeOpen}
    Suppose \(\Phi , \Psi \in MU(X,Y)\)
    and \(U \in \mathscr T_X\) are so that there exist \(f \in \mathrm{sel}(\Phi)\)
    and \(g\in \mathrm{sel}(\Psi)\) with the property that \(f\restriction_U = g\restriction_U\).
    Then \(\Phi\restriction_U = \Psi\restriction_U\).
\end{corollary}
\begin{proof}
    We will show that \(\Phi(x) \subseteq \Psi(x)\) for each \(x \in U\).
    By symmetry, this will establish that \(\Phi(x) = \Psi(x)\) for each \(x \in U\).

    So let \(x \in U\) and \(y \in \Phi(x)\) which, by Theorem \ref{thm:HolaHolyChar}
    means that \(y \in \overline{f}(x)\).
    Now, consider any neighborhood \(V \times W\) of \(\langle x, y \rangle\).
    Without loss of generality, we may assume that \(V \subseteq U\).
    Then we can find \(\langle z, f(z) \rangle \in V \times W\).
    Since \(z \in U\), \(f(z) = g(z)\) and, as
    \(V \times W\) was arbitrary, we see that \(\langle x, y \rangle \in \mathrm{cl}\ \mathrm{gr}(g)\).
    Hence, \(y \in \overline{g}(x) = \Psi(x)\).
\end{proof}
\begin{corollary} \label{cor:UscoAgreeOpen}
    If \(A \subseteq X\) is non-empty, \(U, V \in \mathscr T_X\) are so that
    \(A \subseteq V \subseteq \mathrm{cl}(V) \subseteq U\),
    \(\Phi \in MU(X,Y)\), and \(f \in \mathrm{sel}(\Phi)\), then, for \(y_0 \in Y\), \(g : X \to Y\) defined by
    \[
        g(x) = \begin{cases} y_0, & x \in \mathrm{cl}(X \setminus \mathrm{cl}(V));\\
        f(x), & \mathrm{otherwise},\end{cases}
    \]
    has the property that \(\Psi := \overline{g} \in MU(X,Y)\), \(\langle \Phi, \Psi \rangle \in \mathbf W(A,E)\)
    for any entourage \(E\) of \(Y\), and \(g[X \setminus U] = \{y_0\}\).
\end{corollary}
\begin{proof}
    Note that Theorem \ref{thm:HolaHolyChar} and Lemma \ref{lem:UsefulFunction}
    imply that \(g\) is subcontinuous and quasicontinuous.
    Then \(\Psi := \overline{g} \in MU(X,Y)\).

    Since \(V\) is open, \(\mathrm{cl}(X \setminus \mathrm{cl}(V)) \subseteq X \setminus V\).
    Then \(g(x) = f(x)\) for all \(x \in V\).
    By Corollary \ref{cor:PreUscoAgreeOpen}, we see that \(\Phi\restriction_A = \Psi\restriction_A\)
    as \(A \subseteq V\).
    Also, since \(X \setminus U \subseteq X \setminus \mathrm{cl}(V) \subseteq \mathrm{cl}(X \setminus \mathrm{cl}(V))\),
    we see that \(g[X \setminus U] = \{y_0\}\).
\end{proof}

As a final note in this section,
we offer the following generalization of \cite[Lemma 3.1]{HolaHoly2016} to general uniform spaces.
Recall that Lemma \ref{lem:ClosedEntourages} allows us to restrict our attention
to closed entourages.
\begin{corollary} \label{cor:AlmostLikeContinuous}
    Let \(X\) be a space and \(Y\) be a uniform space.
  If \(\Phi,\Psi \in MU(X,Y)\), a closed entourage \(E\) of \(Y\), and a dense \(D \subseteq X\) are so that
  \(\langle \Phi(x),\Psi(x) \rangle \in hE\) for all \(x \in D\), then \(\langle \Phi(x), \Psi(x) \rangle \in hE\)
  for all \(x \in X\).
\end{corollary}
\begin{proof}
  Define \(F : X \to \wp(Y)\) by \(F(x) = E[\Phi(x)]\).

  We first show that the graph of \(F\) is closed.
  Suppose \(\langle x, y \rangle \in \mathrm{cl}\ \mathrm{gr}(F)\)
  and let \(\langle \langle x_\lambda , y_\lambda \rangle : \lambda \in \Lambda \rangle\)
  be a net in \(\mathrm{gr}(F)\) so that \(\langle x_\lambda, y_\lambda \rangle \to \langle x , y \rangle\).
  Since \(y_\lambda \in E[\Phi(x_\lambda)]\), we can let
  \(w_\lambda \in \Phi(x_\lambda)\) be so that \(y_\lambda \in E[w_\lambda]\).
  Observe that, since \(x_\lambda \to x\) and \(w_\lambda \in \Phi(x_\lambda)\) for each \(\lambda \in \Lambda\),
  by Theorem \ref{thm:HolaHolyChar},
  \(\langle w_\lambda : \lambda \in \Lambda \rangle\) has an accumulation point \(w \in \Phi(x)\).
  Since \(y_\lambda \to y\) and \(w\) is an accumulation point of \(\langle w_\lambda : \lambda \in \Lambda \rangle\),
  \(\langle w, y\rangle\) is an accumulation point of
  \(\langle \langle w_\lambda , y_\lambda \rangle : \lambda \in \Lambda \rangle\).
  Moreover, as \(\langle w_\lambda , y_\lambda \rangle \in E\) for all \(\lambda\in\Lambda\)
  and \(E\) is closed, we see that \(\langle w , y \rangle \in E\).
  Hence, \(y \in E[w] \subseteq E[\Phi(x)] = F(x)\).
  That is, \(\langle x, y \rangle \in \mathrm{gr}(F)\) which establishes that
  \(\mathrm{gr}(F)\) is closed.

  Now, by Theorem \ref{thm:HolaHolyChar}, we can let \(g : D \to Y\) be subcontinuous
  and quasicontinuous so that \(g(x) \in \Psi(x)\) for each \(x\in D\) and \(\Psi = \overline{g}\).
  Since \(\mathrm{gr}(F)\) is closed and \(\mathrm{gr}(g) \subseteq \mathrm{gr}(F)\),
  we see that \(\mathrm{cl}\ \mathrm{gr}(g) \subseteq \mathrm{gr}(F)\).
  That is, \(\Psi(x) \subseteq F(x) = E[\Phi(x)]\) for all \(x \in X\).

  A symmetric argument shows that \(\Phi(x) \subseteq E[\Psi(x)]\) for all \(x \in X\),
  finishing the proof.
\end{proof}

\section{Results}

We first note that the clustering version of Gruenhage's \(W\)-game is equivalent to
an entourage selection game in the realm of topological groups.
Such a result holds, for example, for \(C_{\mathcal A}(X)\) where we define
\(C_{\mathcal A}(X)\) to be the space of continuous real-valued functions on \(X\)
endowed with the topology of uniform convergence on \(\mathcal A\), an ideal of closed subsets of \(X\).
This topology is generated by the uniformity outlined in Definition \ref{def:UniformStructure}.

Recall that a topological group is a (multiplicative) group \(G\) with a topology for which the operations
\(\langle g, h \rangle \mapsto gh\), \(G^2 \to G\), and \(g \mapsto g^{-1}\), \(G \to G\),
are continuous.
Let \(\mathbf i\) be the identity element of \(G\).
Also recall that \[\left\{ \{ \langle g , h \rangle \in G^2 : gh^{-1} \in U \} : U \in \mathscr N_{G,\mathbf i} \right\}\]
is a basis for a uniformity on \(G\).
Let \(\mathscr E_G\) be the set of all entourages of \(G\) with the generated uniformity.
Also let
\[\mathbf \Omega_{\Delta} = \left\{ A \subseteq G^2 :
\{gh^{-1} : \langle g , h \rangle \in A \} \in \Omega_{G,\mathbf i} \right\}.\]
\begin{theorem} \label{thm:Group}
    If \(G\) is a (multiplicative) topological group and \(\mathbf i\) is the identity,
    then, for any \(g \in G\),
    \[\mathsf G_1(\mathscr N_{G,g} , \neg \Omega_{G,g})
    \equiv \mathsf G_1(\mathscr N_{G,\mathbf i} , \neg \Omega_{G,\mathbf i})
    \equiv \mathsf G_1(\mathscr E_G, \neg \mathbf \Omega_\Delta).\]
\end{theorem}
\begin{proof}
    The equivalence
    \[\mathsf G_1(\mathscr N_{G,g} , \neg \Omega_{G,g}) \equiv \mathsf G_1(\mathscr N_{G,\mathbf i} , \neg \Omega_{G,\mathbf i})\]
    can be seen by using the homeomorphism \(x \mapsto gx\), \(G \to G\).

    We first show that \[\mathsf G_1(\mathscr N_{G,\mathbf i} , \neg \Omega_{G,\mathbf i}) \leq_{\mathrm{II}}
    \mathsf G_1(\mathscr E_G, \neg \mathbf \Omega_\Delta).\]
    Define \(\overleftarrow{T}_{\mathrm{I},n} : \mathscr E_G \to \mathscr N_{G,\mathbf i}\) so that
    \[\left\{\langle g, h \rangle \in G^2 : gh^{-1} \in \overleftarrow{T}_{\mathrm{I},n}(E) \right\} \subseteq E.\]

    We now define \(\overrightarrow{T}_{\mathrm{II},n} : G \times \mathscr E_G \to G^2\) in the following way.
    For \(g \in \overleftarrow{T}_{\mathrm{I},n}(E)\), let
    \(\overrightarrow{T}_{\mathrm{II},n}(g,E) = \langle g , \mathbf i \rangle\); otherwise,
    let \(\overrightarrow{T}_{\mathrm{II},n}(g,E) = \langle \mathbf i , \mathbf i \rangle\).
    By definition, we see that \(\overrightarrow{T}_{\mathrm{II},n}(g,E) \in E\) when
    \(g \in \overleftarrow{T}_{\mathrm{I},n}(E)\).

    Suppose \(\langle g_n : n \in \omega \rangle \in \prod_{n\in\omega} \overleftarrow{T}_{\mathrm{I},n}(E_n)\)
    for a sequence \(\langle E_n : n \in \omega \rangle\) of \(\mathscr E_G\) is so that
    \(\{ g_n : n \in \omega \} \not\in \Omega_{G,\mathbf i}\).
    Then let \(U \in \mathscr N_{G,\mathbf i}\) be so that \(U \cap \{ g_n : n \in \omega \} = \emptyset\).
    Notice that \(\overrightarrow{T}_{\mathrm{II},n}(g_n,E_n) \not\in \mathbf \Omega_\Delta\).
    Thus, Theorem \ref{TranslationTheorem} applies.

    We now show that \[\mathsf G_1(\mathscr E_G, \neg \mathbf \Omega_\Delta) \leq_{\mathrm{II}}
    \mathsf G_1(\mathscr N_{G,\mathbf i} , \neg \Omega_{G,\mathbf i}).\]
    Define \(\overleftarrow{T}_{\mathrm{II},n} : \mathscr N_{G,\mathbf i} \to \mathscr E_G\) by the rule
    \[\overleftarrow{T}_{\mathrm{II},n}(U) = \{ \langle g , h \rangle \in G^2 : gh^{-1} \in U \}\]
    and \(\overrightarrow{T}_{\mathrm{II},n} : G^2 \times \mathscr N_{G,\mathbf i} \to G\)
    by \(\overrightarrow{T}_{\mathrm{II},n}(\langle g,h \rangle, U) = gh^{-1}.\)
    Note that, if \(\langle g, h \rangle \in \overleftarrow{T}_{\mathrm{II},n}(U)\),
    \(\overrightarrow{T}_{\mathrm{II},n}(\langle g,h \rangle, U) \in U\).
    So suppose we have
    \[ \langle \langle g_n, h_n \rangle : n \in \omega \rangle \in \prod_{n\in\omega} \overleftarrow{T}_{\mathrm{II},n}(U_n)\]
    for a sequence \(\langle U_n : n \in \omega \rangle\) of \(\mathscr N_{G,\mathbf i}\) so that
    \(\{ \langle g_n, h_n \rangle : n \in \omega \} \not\in \mathbf \Omega_\Delta\).
    Evidently, \(\overrightarrow{T}_{\mathrm{II},n}(\langle g_n,h_n \rangle, U_n) \not\in \Omega_{G,\mathbf i}\).
    Again, Theorem \ref{TranslationTheorem} applies.
\end{proof}

For the remainder of the paper, we will be interested only in real set-valued functions;
so we will let \(MU(X) = MU(X,\mathbb R)\).
We also use, for \(\varepsilon > 0\), \[\Delta_\varepsilon = \{ \langle x, y \rangle \in \mathbb R^2 : |x-y| < \varepsilon \}.\]
For \(A \subseteq X\), we will use \(\mathbf U(A,\varepsilon) = \mathbf U(A,\Delta_\varepsilon)\)
and \(\mathbf W(A,\varepsilon) = \mathbf W(A , \Delta_\varepsilon)\).
For \(Y \subseteq \mathbb R\), let \(\mathbb B(Y,\varepsilon) = \bigcup_{y \in Y} B(y,\varepsilon)\) and
note that
\begin{align*}
    &\mathbf W(A, \varepsilon)\\
    &{\ } = \left\{ \langle \Phi , \Psi \rangle \in \mathbb K(\mathbb R)^X : (\forall x \in A)
    [\Phi(x) \subseteq \mathbb B(\Psi(x),\varepsilon) \wedge \Psi(x) \subseteq \mathbb B(\Phi(x),\varepsilon)] \right\}.
\end{align*}
Then, if \(\mathcal A\) is an ideal of closed subsets of \(X\),
we will use \(MU_{\mathcal A}(X)\) to denote the set \(MU(X)\) with the topology generated by the base
for a uniformity \(\{ \mathbf W(A,\varepsilon) : A \in\mathcal A, \varepsilon > 0\}\).
When \(\mathcal A = [X]^{<\omega}\), we use \(MU_p(X)\) and, when \(\mathcal A = K(X)\),
we use \(MU_k(X)\).
For \(\Phi \in MU(X)\), \(A \subseteq X\), and \(\varepsilon > 0\), we let \([\Phi; A , \varepsilon]
= \mathbf W(A,\varepsilon)[\Phi]\).
We will use \(\mathbf 0\) to denote the function that is constantly \(0\) when dealing
with real-valued functions and the function that is constantly \(\{0\}\) when dealing
with usco maps.

\begin{theorem} \label{thm:FirstEquivalence}
  Let \(X\) be regular and \(\mathcal A\) and \(\mathcal B\) be ideals of closed
  subsets of \(X\).
  Then,
  \begin{enumerate}[label=(\roman*)]
    \item \label{thm:RothbergerFirst}
    \(\mathsf G_1(\mathcal O_X(\mathcal A), \Lambda_X(\mathcal B))
    \leq_{\mathrm{II}} \mathsf G_1(\Omega_{MU_{\mathcal A}(X),\mathbf 0},\Omega_{MU_{\mathcal B}(X),\mathbf 0})\),
    \item \label{thm:RothbergerSecond}
    \(\mathsf G_1(\Omega_{MU_{\mathcal A}(X),\mathbf 0},\Omega_{MU_{\mathcal B}(X),\mathbf 0})
    \leq_{\mathrm{II}} \mathsf G_1(\mathscr D_{MU_{\mathcal A}(X)},\Omega_{MU_{\mathcal B}(X),\mathbf 0})\), and
    \item \label{thm:RothbergerThird}
    if \(X\) is \(\mathcal A\)-normal,
    \(\mathsf G_1(\mathscr D_{MU_{\mathcal A}(X)},\Omega_{MU_{\mathcal B}(X),\mathbf 0})
    \leq_{\mathrm{II}} \mathsf G_1(\mathcal O_X(\mathcal A), \Lambda_X(\mathcal B))\).
  \end{enumerate}
  Thus, if \(X\) is \(\mathcal A\)-normal, the three games are equivalent.
\end{theorem}
\begin{proof}
  We first address \ref{thm:RothbergerFirst}.
  Fix some \(\mathscr U_0 \in \mathcal O_X(\mathcal A)\) and let \(W_{\Phi,n} = \Phi^\leftarrow[(-2^{-n},2^{-n})]\)
  for \(\Phi \in MU(X)\) and \(n \in \omega\).
  Define \(\overleftarrow{T}_{\mathrm{I},n} : \Omega_{MU_{\mathcal A}(X),\mathbf 0} \to \mathcal O_X(\mathcal A)\)
  by the rule
  \[\overleftarrow{T}_{\mathrm{I},n}(\mathscr F) =
  \begin{cases}
    \{ W_{\Phi,n} : \Phi \in \mathscr F \}, & (\forall \Phi \in \mathscr F)\ W_{\Phi,n} \neq X;\\
    \mathscr U_0, & \mathrm{otherwise}.
  \end{cases}\]
  To see that \(\overleftarrow{T}_{\mathrm{I},n}\) is defined, let \(\mathscr F \in \Omega_{MU_{\mathcal A}(X),\mathbf 0}\)
  be so that \(W_{\Phi,n} \neq X\) for every \(\Phi \in \mathscr F\).
  Let \(A \in \mathcal A\) be arbitrary and choose \(\Phi \in [\mathbf 0;A,2^{-n}] \cap \mathscr F\).
  It follows that \(A \subseteq W_{\Phi,n}\).
  Hence, \(\overleftarrow{T}_{\mathrm{I},n}(\mathscr F) \in \mathcal O_X(\mathcal A)\).

  We now define
  \[\overrightarrow{T}_{\mathrm{II},n} : \mathscr T_X \times \Omega_{MU_{\mathcal A}(X),\mathbf 0} \to MU(X)\]
  in the following way.
  Let \[\mathfrak T_n = \{ \mathscr F \in \Omega_{MU_{\mathcal A}(X),\mathbf 0} : (\exists \Phi \in \mathscr F)\ W_{\Phi,n} = X \}\]
  and \(\mathfrak T_n^\star = \Omega_{MU_{\mathcal A}(X),\mathbf 0} \setminus \mathfrak T_n\).
  For each \(\langle U ,\mathscr F \rangle \in \mathscr T_X \times \mathfrak T_n\), let
  \(\overrightarrow{T}_{\mathrm{II},n}(U,\mathscr F) \in \mathscr F\) be so that
  \(W_{\overrightarrow{T}_{\mathrm{II},n}(U,\mathscr F),n} = X\).
  For \(\langle U, \mathscr F \rangle \in \mathscr T_X \times \mathfrak T_n^\star\) so that
  \(U \in \overleftarrow{T}_{\mathrm{I},n}(\mathscr F)\), let \(\overrightarrow{T}_{\mathrm{II},n}(U,\mathscr F) \in \mathscr F\)
  be so that \(U = W_{\overrightarrow{T}_{\mathrm{II},n}(U,\mathscr F),n}\).
  For \(\langle U, \mathscr F \rangle \in \mathscr T_X \times \mathfrak T_n^\star\) so that
  \(U \not\in \overleftarrow{T}_{\mathrm{I},n}(\mathscr F)\), let \(\overrightarrow{T}_{\mathrm{II},n}(U,\mathscr F) = \mathbf 0\).
  By construction, if \(U \in \overleftarrow{T}_{\mathrm{I},n}(\mathscr F)\),
  then \(\overrightarrow{T}_{\mathrm{II},n}(U,\mathscr F) \in \mathscr F\).

  To finish this application of Theorem \ref{TranslationTheorem}, assume that we have
  \[\langle U_n : n \in \omega \rangle \in \prod_{n\in\omega} \overleftarrow{T}_{\mathrm{I},n}(\mathscr F_n)\]
  for some sequence \(\langle \mathscr F_n : n \in \omega \rangle\) of \(\Omega_{MU_{\mathcal B}(X),\mathbf 0}\)
  so that \(\{ U_n : n \in \omega \} \in \Lambda_X(\mathcal B)\).
  For each \(n \in \omega\), let \(\Phi_n = \overrightarrow{T}_{\mathrm{II},n}(U_n,\mathscr F_n)\).
  Now, let \(B \in \mathcal B\) and \(\varepsilon > 0\) be arbitrary.
  Choose \(n \in \omega\) so that \(2^{-n} < \varepsilon\) and \(B \subseteq U_n\).
  If \(\mathscr F_n \in \mathfrak T_n\), then \(\Phi_n\) has the property that \(X = \Phi_n^\leftarrow[(-2^{-n},2^{-n})]\);
  hence, \(\Phi_n \in [\mathbf 0; B, \varepsilon]\).
  Otherwise, \(B \subseteq U_n = \Phi_n^\leftarrow[(-2^{-n},2^{-n})]\) which also implies that
  \(\Phi_n \in [\mathbf 0; B , \varepsilon]\).
  Thus, \(\{ \Phi_n : n \in \omega \} \in \Omega_{MU_{\mathcal B}(X),\mathbf 0}\).

  \ref{thm:RothbergerSecond} holds since \(\mathscr D_{MU_{\mathcal A}(X)} \subseteq
  \Omega_{MU_{\mathcal A}(X),\mathbf 0}\).

  Lastly, we address \ref{thm:RothbergerThird}.
  We define
  \(\overset{\leftarrow}{T}_{\mathrm I,n} : \mathcal O_X(\mathcal A) \to \mathscr D_{MU_{\mathcal A}(X)}\)
  by the rule \[\overset{\leftarrow}{T}_{\mathrm I,n}(\mathscr U)
  = \{ \Phi \in MU(X) : (\exists U \in \mathscr U)(\exists f \in \mathrm{sel}(\Phi))\ f[X \setminus U] = \{1\} \}.\]
  To see that \(\overset{\leftarrow}{T}_{\mathrm I,n}\) is defined, let
  \(\mathscr U \in \mathcal O_X(\mathcal A)\) and consider a basic open set
  \([\Phi ; A , \varepsilon ]\).
  Then let \(U \in \mathscr U\) be so that \(A \subseteq U\) and, by \(\mathcal A\)-normality,
  let \(V\) be open so that \(A \subseteq V \subseteq \mathrm{cl}(V) \subseteq U\).
  Define \(f : X \to \mathbb R\) by the rule
  \[f(x) =
  \begin{cases}
    1, & x \in \mathrm{cl}(X \setminus \mathrm{cl}(V));\\
    \max \Phi(x), & \mathrm{otherwise}.
  \end{cases}
  \]
  By Corollary \ref{cor:UscoAgreeOpen},
  \(\overline{f} \in [\Phi; A, \varepsilon] \cap \overset{\leftarrow}{T}_{\mathrm I,n}(\mathscr U)\).
  Hence, \(\overset{\leftarrow}{T}_{\mathrm I,n}(\mathscr U) \in \mathscr D_{MU_{\mathcal A}(X)}\).

  We define \(\overset{\rightarrow}{T}_{\mathrm{II},n} : MU(X) \times \mathcal O_X(\mathcal A) \to \mathscr T_X\)
  in the following way.
  Fix some \(U_0 \in \mathscr T_X\).
  For \(\langle \Phi , \mathscr U \rangle \in MU(X) \times \mathcal O_X(\mathcal A)\),
  if \[\{ U \in \mathscr U : (\exists f \in \mathrm{sel}(\Phi))\ f[X \setminus U] = \{1\} \} \neq \emptyset,\]
  let \(\overset{\rightarrow}{T}_{\mathrm{II},n}(\Phi, \mathscr U) \in \mathscr U\) be so that
  there exists \(f \in \mathrm{sel}(\Phi)\) with the property that
  \(f[X \setminus \overset{\rightarrow}{T}_{\mathrm{II},n}(\Phi, \mathscr U)] = \{1\};\)
  otherwise, let \(\overset{\rightarrow}{T}_{\mathrm{II},n}(\Phi, \mathscr U) = U_0\).
  By construction, if \(\Phi \in \overset{\leftarrow}{T}_{\mathrm I,n}(\mathscr U)\),
  then \(\overset{\rightarrow}{T}_{\mathrm{II},n}(\Phi, \mathscr U) \in \mathscr U\).

  Suppose we have
  \[\langle \Phi_n : n \in \omega \rangle \in \prod_{n\in\omega} \overset{\leftarrow}{T}_{\mathrm I,n}(\mathscr U_n)\]
  for a sequence \(\langle \mathscr U_n : n \in \omega\rangle\) of \(\mathcal O_X(\mathcal A)\)
  with the property that \(\{ \Phi_n : n \in \omega \} \in \Omega_{MU_{\mathcal B}(X),\mathbf 0}\).
  For each \(n \in \omega\), let \(U_n = \overset{\rightarrow}{T}_{\mathrm{II},n}(\Phi_n, \mathscr U_n)\).
  Since \(\mathcal B\) is an ideal of sets, we need only show that \(\langle U_n : n \in \omega \rangle\)
  is a \(\mathcal B\)-cover.
  So let \(B \in \mathcal B\) be arbitrary and let \(n \in \omega\) be so that
  \(\Phi_n \in [\mathbf 0; B, 1]\).
  Then we can let \(f \in \mathrm{sel}(\Phi_n)\) be so that \(f[X \setminus U_n] = \{1\}\).
  Since \(\Phi_n \in [\mathbf 0;B,1]\), we see that, for each \(x \in B\), \(f(x) \in \Phi_n(x) \subseteq (-1,1)\).
  Hence, \(B \cap (X \setminus U_n) = \emptyset\), which is to say that \(B \subseteq U_n\).
  So Theorem \ref{TranslationTheorem} applies.
\end{proof}

We now establish some relationships between these games and games on the space of continuous real-valued functions.
\begin{corollary} \label{cor:Rothberger}
    Let \(\mathcal A\) and \(\mathcal B\) be ideals of closed subsets of \(X\) and suppose that
    \(X\) is \(\mathcal A\)-normal.
    Then
    \begin{align*}
        \mathcal G := \mathsf G_1(\mathcal O_X(\mathcal A), \mathcal O_X(\mathcal B))
        &\equiv \mathsf G_1(\Omega_{MU_{\mathcal A}(X),\mathbf 0},\Omega_{MU_{\mathcal B}(X),\mathbf 0})\\
        &\equiv \mathsf G_1(\mathscr D_{MU_{\mathcal A}(X)},\Omega_{MU_{\mathcal B}(X),\mathbf 0}),
    \end{align*}
    \begin{align*}
        \mathcal H := \mathsf G_1(\mathscr N_X[\mathcal A], \neg \mathcal O_X(\mathcal B))
        &\equiv \mathsf G_1(\mathscr N_{MU_{\mathcal A}(X),\mathbf 0}, \neg \Omega_{MU_{\mathcal B}(X),\mathbf 0})\\
        &\equiv \mathsf G_1(\mathscr T_{MU_{\mathcal A}(X)}, \neg \Omega_{MU_{\mathcal B}(X),\mathbf 0}),
    \end{align*}
    and \(\mathcal G\) is dual to \(\mathcal H\).
    If \(X\) is functionally \(\mathcal A\)-normal, then
    \begin{align*}
        \mathsf G_1(\mathcal O_X(\mathcal A), \mathcal O_X(\mathcal B))
        &\equiv \mathsf G_1(\Omega_{MU_{\mathcal A}(X),\mathbf 0},\Omega_{MU_{\mathcal B}(X),\mathbf 0})\\
        &\equiv \mathsf G_1(\Omega_{C_{\mathcal A}(X),\mathbf 0},\Omega_{C_{\mathcal B}(X),\mathbf 0})\\
        &\equiv \mathsf G_1(\mathscr D_{MU_{\mathcal A}(X)},\Omega_{MU_{\mathcal B}(X),\mathbf 0})\\
        &\equiv \mathsf G_1(\mathscr D_{C_{\mathcal A}(X)},\Omega_{C_{\mathcal B}(X),\mathbf 0})
    \end{align*}
    and
    \begin{align*}
        \mathsf G_1(\mathscr N_X[\mathcal A], \neg \mathcal O_X(\mathcal B))
        &\equiv \mathsf G_1(\mathscr N_{MU_{\mathcal A}(X),\mathbf 0}, \neg \Omega_{MU_{\mathcal B}(X),\mathbf 0})\\
        &\equiv \mathsf G_1(\mathscr N_{C_{\mathcal A}(X),\mathbf 0}, \neg \Omega_{C_{\mathcal B}(X),\mathbf 0})\\
        &\equiv \mathsf G_1(\mathscr T_{MU_{\mathcal A}(X)}, \neg \Omega_{MU_{\mathcal B}(X),\mathbf 0})\\
        &\equiv \mathsf G_1(\mathscr T_{C_{\mathcal A}(X)}, \neg \Omega_{C_{\mathcal B}(X),\mathbf 0}).
    \end{align*}
\end{corollary}
\begin{proof}
    Apply Theorem \ref{thm:FirstEquivalence}, Lemmas \ref{lem:IdealLargeCovers}, \ref{lem:BasicEquivDual}, and \ref{lem:DualGames},
    and \cite[Cor. 14]{CHContinuousFunctions}.
\end{proof}
In \cite[Thm. 31]{CHHyperspaces}, inspired by Li \cite{Li2016}, the game
\(\mathsf G_1(\mathcal O_X(\mathcal A) , \mathcal O_X(\mathcal B))\) is shown to be equivalent
to the selective separability game on certain hyperspaces of \(X\),
which we only note in passing here for the interested reader.

Recall that a subset \(A\) of a topological space is \emph{sequentially compact} if every sequence in
\(A\) has a subsequence which converges to a point of \(A\).
\begin{lemma} \label{lem:SequentiallyCompact}
    Suppose \(\Phi \in MU(X)\) and that \(A \subseteq X\) is sequentially compact.
    Then \(\Phi[A]\) is bounded.
\end{lemma}
\begin{proof}
    Suppose \(\Phi : X \to \mathbb K(\mathbb R)\) is unbounded on \(A \subseteq X\) which is sequentially compact.
    For each \(n \in \omega\), let \(x_n \in A\) be so that there is some \(y \in \Phi(x_n)\) with \(|y| \geq n\).
    Then let \(y_n \in \Phi(x_n)\) be so that \(|y_n| \geq n\) and define \(f : X \to \mathbb R\) to be a selection
    of \(\Phi\) so that \(f(x_n) = y_n\) for \(n \in \omega\).
    Since \(A\) is sequentially compact, we can find \(x \in A\) and a subsequence \(\langle x_{n_k} : k \in \omega \rangle\)
    so that \(x_{n_k} \to x\).
    Notice that \(\langle f(x_{n_k}) : k \in \omega \rangle\) does not have an accumulation point.
    Therefore \(f\) is not subcontinuous and, by Theorem \ref{thm:HolaHolyChar}, \(\Phi\) is not a minimal usco map.
\end{proof}

\begin{theorem} \label{thm:SecondTheorem}
    Let \(\mathcal A\) and \(\mathcal B\) be ideals of closed subsets of \(X\).
    If \(X\) is \(\mathcal A\)-normal and \(\mathcal B\) consists of sequentially compact sets, then
    \[
        \mathsf G_1(\mathscr{N}_X[\mathcal A], \neg \Lambda_X(\mathcal B))
        \leq_{\mathrm{II}} \mathsf G_1(\mathscr{T}_{MU_{\mathcal A}(X)}, \mathrm{CD}_{MU_{\mathcal B}(X)}).
    \]
    Consequently,
    \begin{align*}
        \mathsf G_1(\mathscr N_X[\mathcal A], \neg \mathcal O_X(\mathcal B))
        &\equiv \mathsf G_1(\mathscr N_{MU_{\mathcal A}(X),\mathbf 0}, \neg \Omega_{MU_{\mathcal B}(X),\mathbf 0})\\
        &\equiv \mathsf G_1(\mathscr T_{MU_{\mathcal A}(X)}, \neg \Omega_{MU_{\mathcal B}(X),\mathbf 0})\\
        &\equiv \mathsf G_1(\mathscr{T}_{MU_{\mathcal A}(X)}, \mathrm{CD}_{MU_{\mathcal B}(X)}).
    \end{align*}
\end{theorem}
\begin{proof}
    Let \(\pi_1 : MU(X) \times \mathcal A \times \mathbb R \to MU(X)\),
    \(\pi_2 : MU(X) \times \mathcal A \times \mathbb R \to \mathcal A\),
    and \(\pi_3 : MU(X) \times \mathcal A \times \mathbb R \to \mathbb R\) be the
    standard coordinate projection maps.
    Define a choice function \(\gamma : \mathscr T_{MU_{\mathcal A}(X)} \to MU(X) \times \mathcal A \times \mathbb R\)
    so that \[[\pi_1(\gamma(W)); \pi_2(\gamma(W)) , \pi_3(\gamma(W))] \subseteq W.\]
    Let \(\Psi_{W} = \pi_1(\gamma(W))\) and \(A_W = \pi_2(\gamma(W))\).
    Now we define \(\overleftarrow{T}_{\mathrm{I},n} : \mathscr T_{MU_{\mathcal A}(X)} \to \mathscr N_X[\mathcal A]\)
    by \(\overleftarrow{T}_{\mathrm{I},n}(W) = \mathscr N_X(A_W)\).

    We now define
    \(\overrightarrow{T}_{\mathrm{II},n} : \mathscr T_X \times \mathscr T_{MU_{\mathcal A}(X)} \to MU(X)\)
    in the following way.
    For \(A \in \mathcal A\) and \(U \in \mathscr N_X(A)\), let \(V_{A,U}\) be open so that
    \[A \subseteq V_{A,U} \subseteq \mathrm{cl}(V_{A,U}) \subseteq U.\]
    For \(W \in \mathscr T_{MU_{\mathcal A}(X)}\) and \(U \in \overleftarrow{T}_{\mathrm{I},n}(W)\),
    define \(f_{W,U,n} : X \to \mathbb R\) by the rule
    \[f_{W,U,n}(x) =
    \begin{cases}
        n, & x \in \mathrm{cl}(X \setminus \mathrm{cl}(V_{A_W,U}));\\
        \max \Psi_W(x), & \mathrm{otherwise}.\\
    \end{cases}
    \]
    Then we set
    \[\overrightarrow{T}_{\mathrm{II},n}(U,W) =
    \begin{cases}
        \overline{f}_{W,U,n}, & U \in \overleftarrow{T}_{\mathrm{I},n}(W);\\
        \mathbf 0, & \text{otherwise}.
    \end{cases}
    \]
    By Corollary \ref{cor:UscoAgreeOpen}, \(\overrightarrow{T}_{\mathrm{II},n}(U,W) \in MU(X)\) and,
    if \(U \in \overleftarrow{T}_{\mathrm{I},n}(W)\),
    \[\overrightarrow{T}_{\mathrm{II},n}(U,W) \in [\Psi_W;A_W,\pi_3(\gamma(W))] \subseteq W.\]

    Suppose we have a sequence
    \[\langle U_n : n \in \omega \rangle \in \prod_{n\in\omega} \overleftarrow{T}_{\mathrm{I},n}(W_n)\]
    for a sequence \(\langle W_n : n \in \omega \rangle\) of \(\mathscr T_{MU_{\mathcal A}(X)}\)
    so that \(\{ U_n : n \in \omega \} \not\in \Lambda_X(\mathcal B)\).
    Let \(\Phi_n = \overrightarrow{T}_{\mathrm{II},n}(U_n,W_n)\) for each \(n \in \omega\).
    We can find \(N \in \omega\) and \(B \in \mathcal B\) so that, for every \(n \geq N\),
    \(B \not\subseteq U_n\).
    Now, suppose \(\Phi \in MU(X) \setminus \{\Phi_n : n \in \omega\}\) is arbitrary.
    By Lemma \ref{lem:SequentiallyCompact}, \(\Phi[B]\) is bounded, so let \(M > \sup |\Phi[B]|\)
    and \(n \geq \max\{N,M+1\}\).
    Now, for \(x \in B \setminus U_n\), note that \(n \in \Phi_n(x)\) and that, for \(y \in \Phi(x)\),
    \[y \leq \sup |\Phi[B]| < M \leq n-1 \implies y-n < -1 \implies |y-n| > 1.\]
    In particular, \(\Phi_n(x) \not\subseteq \mathbb B(\Phi(x),1)\) which establishes that
    \(\Phi_n \not\in [\Phi; B, 1]\).
    Hence, \(\{\Phi_n : n \in \omega \}\) is closed and discrete and Theorem \ref{TranslationTheorem} applies.

    For what remains, observe that
    \[\mathsf G_1(\mathscr{T}_{MU_{\mathcal A}(X)}, \mathrm{CD}_{MU_{\mathcal B}(X)})
    \leq_{\mathrm{II}} \mathsf G_1(\mathscr T_{MU_{\mathcal A}(X)}, \neg \Omega_{MU_{\mathcal B}(X),\mathbf 0})\]
    since, if Two can produce a closed discrete set, then Two can avoid clustering around \(\mathbf 0\).
    Hence, by Corollary \ref{cor:Rothberger} we obtain that
    \begin{align*}
        \mathsf G_1(\mathscr{N}_X[\mathcal A], \neg \mathcal O_X(\mathcal B))
        &= \mathsf G_1(\mathscr{N}_X[\mathcal A], \neg \Lambda_X(\mathcal B))\\
        &\leq_{\mathrm{II}} \mathsf G_1(\mathscr{T}_{MU_{\mathcal A}(X)}, \mathrm{CD}_{MU_{\mathcal B}(X)})\\
        &\leq_{\mathrm{II}} \mathsf G_1(\mathscr T_{MU_{\mathcal A}(X)}, \neg \Omega_{MU_{\mathcal B}(X),\mathbf 0})\\
        &\equiv \mathsf G_1(\mathscr{N}_X[\mathcal A], \neg \mathcal O_X(\mathcal B)).
    \end{align*}
    This complete the proof.
\end{proof}

\begin{corollary} \label{cor:PointOpen}
    Let \(\mathcal A\) and \(\mathcal B\) be ideals of closed subsets of \(X\).
    If \(X\) is functionally \(\mathcal A\)-normal and \(\mathcal B\) consists of
    sequentially compact sets, then
    \begin{align*}
        \mathsf G_1(\mathscr N_X[\mathcal A], \neg \mathcal O_X(\mathcal B))
        &\equiv \mathsf G_1(\mathscr N_{MU_{\mathcal A}(X),\mathbf 0}, \neg \Omega_{MU_{\mathcal B}(X),\mathbf 0})\\
        &\equiv \mathsf G_1(\mathscr N_{C_{\mathcal A}(X),\mathbf 0}, \neg \Omega_{C_{\mathcal B}(X),\mathbf 0})\\
        &\equiv \mathsf G_1(\mathscr T_{MU_{\mathcal A}(X)}, \neg \Omega_{MU_{\mathcal B}(X),\mathbf 0})\\
        &\equiv \mathsf G_1(\mathscr T_{C_{\mathcal A}(X)}, \neg \Omega_{C_{\mathcal B}(X),\mathbf 0})\\
        &\equiv \mathsf G_1(\mathscr{T}_{MU_{\mathcal A}(X)}, \mathrm{CD}_{MU_{\mathcal B}(X)})\\
        &\equiv \mathsf G_1(\mathscr{T}_{C_{\mathcal A}(X)}, \mathrm{CD}_{C_{\mathcal B}(X)}).
    \end{align*}
\end{corollary}
\begin{proof}
    Apply Theorem \ref{thm:SecondTheorem} and \cite[Cor. 15]{CHContinuousFunctions}.
\end{proof}
We now offer some relationships related to Gruenhage's \(W\)-games.
\begin{proposition} \label{prop:ConvergenceGames}
    Let \(\mathcal A\) and \(\mathcal B\) be ideals of closed subsets of \(X\).
    Then
    \begin{enumerate}[label=(\roman*)]
        \item \label{prop:ConvergenceGamesA}
        \(\mathsf G_1(\mathscr N_{MU_{\mathcal A}(X),\mathbf 0} , \neg \Omega_{MU_{\mathcal B}(X),\mathbf 0})
        \leq_{\mathrm{II}} \mathsf G_1(\mathscr N_{MU_{\mathcal A}(X),\mathbf 0} , \neg \Gamma_{MU_{\mathcal B}(X),\mathbf 0})\) and
        \item \label{prop:ConvergenceGamesB}
        \(\mathsf G_1(\mathscr N_{MU_{\mathcal A}(X),\mathbf 0} , \neg \Gamma_{MU_{\mathcal B}(X),\mathbf 0})
        \leq_{\mathrm{II}} \mathsf G_1(\mathscr N_X[\mathcal A], \neg \Gamma_X(\mathcal B))\).
    \end{enumerate}
\end{proposition}
\begin{proof}
    \ref{prop:ConvergenceGamesA} is evident since, if Two can avoid clustering at \(\mathbf 0\), they can surely
    avoid converging to \(\mathbf 0\).

    \ref{prop:ConvergenceGamesB}
    Fix \(U_0 \in \mathscr T_X\) and
    define \(\overleftarrow{T}_{\mathrm{I},n} : \mathscr N_X[\mathcal A] \to \mathscr N_{MU_{\mathcal A}(X),\mathbf 0}\)
    by \(\overleftarrow{T}_{\mathrm{I},n}(\mathscr N_X(A)) = \left[\mathbf 0; A , 2^{-n}\right]\).
    Then define \(\overrightarrow{T}_{\mathrm{II},n} : MU(X) \times \mathscr N_X[\mathcal A] \to \mathscr T_X\) by
    \(\overrightarrow{T}_{\mathrm{II},n}(\Phi, \mathscr N_X(A)) = \Phi^{\leftarrow}\left[\left(-2^{-n},2^{-n}\right)\right]\).
    Note that, if \(\Phi \in \left[\mathbf 0; A , 2^{-n}\right] = \overleftarrow{T}_{\mathrm{I},n}(\mathscr N_X(A)),\)
    then \(A \subseteq \Phi^\leftarrow\left[\left(-2^{-n},2^{-n}\right)\right]\), which establishes that
    \(\overrightarrow{T}_{\mathrm{II},n}(\Phi, \mathscr N_X(A)) \in \mathscr N_X(A)\).

    Suppose we have
    \[\left\langle \Phi_n : n \in \omega \right\rangle \in \prod_{n\in\omega} \overleftarrow{T}_{\mathrm{I},n}(\mathscr N_X(A_n))\]
    for a sequence \(\langle A_n : n \in \omega \rangle\) of \(\mathcal A\)
    so that \(\langle \Phi_n : n \in \omega \rangle \not\in \Gamma_{MU_{\mathcal B}(X),\mathbf 0}\).
    Then we can find \(B \in \mathcal B\), \(\varepsilon > 0\), and \(N \in \omega\) so that \(2^{-N} < \varepsilon\) and,
    for all \(n \geq N\), \(\Phi_n \not\in [\mathbf 0; B, \varepsilon]\).

    To finish this application of Theorem \ref{TranslationTheorem}, we need to show that \[B \not\subseteq
    \overrightarrow{T}_{\mathrm{II},n}(\Phi_n, \mathscr N_X(A_n))\] for all \(n \geq N\).
    So let \(n \geq N\) and note that, since \(\Phi_n \not\in [\mathbf 0; B, \varepsilon]\),
    there is some \(x \in B\) and \(y \in \Phi_n(x)\) so that \(|y| \geq \varepsilon > 2^{-N} \geq 2^{-n}\).
    That is, \(\Phi_n(x) \not\subseteq (-2^{-n} , 2^{-n})\) and so
    \(x \not\in \overrightarrow{T}_{\mathrm{II},n}(\Phi_n, \mathscr N_X(A_n))\).
    This finishes the proof.
\end{proof}

Though particular applications of Corollaries \ref{cor:Rothberger} and \ref{cor:PointOpen}
abound, we record a few that capture the general spirit using ideals of usual interest
after recalling some other facts and some names for particular selection principles.
We suppress the relationships with \(C_{\mathcal A}(X)\) in the following applications
in the interest of space.
\begin{definition}
    We identify some particular selection principles by name.
    \begin{itemize}
        \item
        \(\mathsf S_1(\Omega_{X,x}, \Omega_{X,x})\) is known as the \textdefn{strong countable fan-tightness property
        for \(X\) at \(x\)}.
        \item
        \(\mathsf S_1(\mathscr D_{X}, \Omega_{X,x})\) is known as the \textdefn{strong countable dense fan-tightness property
        for \(X\) at \(x\)}.
        \item
        \(\mathsf S_1(\mathscr T_X, \mathrm{CD}_X)\) is known as the \textdefn{discretely selective property for \(X\)}.
        \item
        We refer to \(\mathsf S_1(\Omega_X,\Omega_X)\) as the \textdefn{\(\omega\)-Rothberger property}
        and \(\mathsf S_1(\mathcal K_X,\mathcal K_X)\) as the \textdefn{\(k\)-Rothberger property}.
    \end{itemize}
\end{definition}

\begin{definition}
    For a partially ordered set \((\mathbb P, \leq)\) and collections \(\mathcal A , \mathcal B \subseteq \mathbb P\)
    so that, for every \(B \in \mathcal B\), there exists some \(A \in \mathcal A\) with \(B \subseteq A\),
    we define the \textdefn{cofinality of \(\mathcal A\) relative to \(\mathcal B\)} by
    \[\mathrm{cof}(\mathcal A; \mathcal B, \leq)
    = \min \{ \kappa \in \mathrm{CARD} : (\exists \mathscr F \in [\mathcal A]^\kappa)(\forall B \in \mathcal B)(\exists A \in \mathscr F)
    \ B \subseteq A\}\]
    where \(\mathrm{CARD}\) is the class of cardinals.
\end{definition}
\begin{lemma} \label{lem:TelgarskyCofinal}
    Let \(\mathcal A, \mathcal B \subseteq \wp^+(X)\) for a space \(X\).

    As long as \(X\) is \(T_1\),
    \[\mathrm{I} \underset{\mathrm{pre}}{\uparrow} \mathsf G_1(\mathscr N_X[\mathcal A] , \neg \mathcal O_X(\mathcal B))
    \iff \mathrm{cof}(\mathcal A;\mathcal B, \subseteq) \leq \omega.\]
    (See {\cite{GerlitsNagy,TkachukCpBook}, and \cite[Lemma 23]{CHContinuousFunctions}}.)

    If \(\mathcal A\) consists of \(G_\delta\) sets,
    \begin{align*}
        \mathrm{I} \uparrow \mathsf G_1(\mathscr N_X[\mathcal A] , \neg \mathcal O_X(\mathcal B))
        &\iff \mathrm{I} \underset{\mathrm{pre}}{\uparrow} \mathsf G_1(\mathscr N_X[\mathcal A] , \neg \mathcal O_X(\mathcal B))\\
        &\iff \mathrm{cof}(\mathcal A;\mathcal B, \subseteq) \leq \omega.
    \end{align*}
    (See {\cite{Galvin1978,Telgarsky1975} and \cite[Lemma 24]{CHContinuousFunctions}}.)
\end{lemma}
Observe that Lemma \ref{lem:TelgarskyCofinal} informs us that, for a \(T_1\) space \(X\),
\begin{itemize}
    \item
    \(\mathrm{I} \underset{\mathrm{pre}}{\uparrow} \mathsf G_1(\mathbb P_X, \neg\mathcal O_X)\) if and only if \(X\) is countable,
    \item
    \(\mathrm{I} \underset{\mathrm{pre}}{\uparrow} \mathsf G_1(\mathscr N_X[K(X)], \neg\mathcal O_X)\) if and only if
    \(X\) is \(\sigma\)-compact, and
    \item
    \(\mathrm{I} \underset{\mathrm{pre}}{\uparrow} \mathsf G_1(\mathscr N_X[K(X)], \neg\mathcal K_X)\) if and only if \(X\) is hemicompact.
\end{itemize}
\begin{corollary}
    For an ideal \(\mathcal A\) of closed subsets of a \(T_1\) space \(X\), \(\mathrm{cof}(\mathcal A; \mathcal A, \subseteq) \leq \omega\)
    if and only if \(MU_{\mathcal A}(X)\) is metrizable.
\end{corollary}
\begin{proof}
    If \(\{ A_n : n \in \omega \} \subseteq \mathcal A\) is so that, for every \(A \in \mathcal A\), there is an \(n\in\omega\)
    with \(A \subseteq A_n\), notice that the family \(\{ \mathbf W(A_n,2^{-m}) : n,m \in \omega\}\) is a countable
    base for the uniformity on \(MU_{\mathcal A}(X)\); so Theorem \ref{thm:UniformMetrizability} demonstrates
    that \(MU_{\mathcal A}(X)\) is metrizable.

    Now, suppose \(MU_{\mathcal A}(X)\) is metrizable, which implies that \(MU_{\mathcal A}(X)\) is first-countable.
    Using a descending countable basis at \(\mathbf 0\), we see that
    \[\mathrm{I} \underset{\mathrm{pre}}{\uparrow} \mathsf G_1(\mathscr N_{MU_{\mathcal A}(X),\mathbf 0} ,
    \neg \Gamma_{MU_{\mathcal A}(X),\mathbf 0}),\]
    and, in particular,
    \[\mathrm{I} \underset{\mathrm{pre}}{\uparrow} \mathsf G_1(\mathscr N_{MU_{\mathcal A}(X),\mathbf 0} ,
    \neg \Omega_{MU_{\mathcal A}(X),\mathbf 0}).\]
    By Corollary \ref{cor:PointOpen}, we see that
    \[\mathrm{I} \underset{\mathrm{pre}}{\uparrow} \mathsf G_1(\mathscr N_X[\mathcal A], \neg \mathcal O_X(\mathcal A)).\]
    So, by Lemma \ref{lem:TelgarskyCofinal}, \(\mathrm{cof}(\mathcal A; \mathcal A, \subseteq) \leq \omega\).
\end{proof}
As a particular consequence of this, we see that
\begin{corollary} \label{cor:ParticularMetrizability}
    For any regular space \(X\), the following are equivalent.
    \begin{enumerate}[label=(\roman*)]
        \item
        \(X\) is countable.
        \item
        \(MU_p(X)\) is metrizable.
        \item
        \(MU_p(X)\) is not discretely selective.
        \item
        \(\mathrm{II} \underset{\mathrm{mark}}{\uparrow} \mathsf G_1(\Omega_X,\Omega_X)\).
        \item
        \(\mathrm{II} \underset{\mathrm{mark}}{\uparrow} \mathsf G_1(\Omega_{MU_p(X),\mathbf 0},\Omega_{MU_p(X),\mathbf 0})\).
        \item
        \(\mathrm{II} \underset{\mathrm{mark}}{\uparrow} \mathsf G_1(\mathscr D_{MU_p(X)},\Omega_{MU_p(X),\mathbf 0})\).
    \end{enumerate}
    Also, the following are equivalent.
    \begin{enumerate}[label=(\roman*)]
        \item
        \(X\) is hemicompact.
        \item
        \(\mathbb K(X)\) is hemicompact. (See \cite[Thm. 3.22]{CHCompactOpen}.)
        \item
        \(MU_k(X)\) is metrizable. (See \cite[Cor. 4.5]{HolaHoly2022}.)
        \item
        \(MU_k(X)\) is not discretely selective.
        \item
        \(\mathrm{II} \underset{\mathrm{mark}}{\uparrow} \mathsf G_1(\mathcal K_X,\mathcal K_X)\).
        \item
        \(\mathrm{II} \underset{\mathrm{mark}}{\uparrow} \mathsf G_1(\Omega_{MU_k(X),\mathbf 0},\Omega_{MU_k(X),\mathbf 0})\).
        \item
        \(\mathrm{II} \underset{\mathrm{mark}}{\uparrow} \mathsf G_1(\mathscr D_{MU_k(X)},\Omega_{MU_k(X),\mathbf 0})\).
    \end{enumerate}
\end{corollary}
Before the next corollary, we recall Tkachuk's strategy strengthening for the generalized point-open games.
\begin{theorem}[See {\cite[Cor. 11]{CHContinuousFunctions}} and {\cite{TkachukCpBook}}] \label{thm:TkachukStrengthening}
    Let \(\mathcal A\) and \(\mathcal B\) be ideals of closed subsets of \(X\).
    Then
    \[\mathrm{I} \uparrow \mathsf G_1(\mathscr N_X[\mathcal A], \neg \mathcal O_X(\mathcal B))
    \iff \mathrm{I} \uparrow \mathsf G_1(\mathscr N_X[\mathcal A], \neg \Gamma_X(\mathcal B))\]
    and
    \[\mathrm{I} \underset{\mathrm{pre}}{\uparrow} \mathsf G_1(\mathscr N_X[\mathcal A], \neg \mathcal O_X(\mathcal B))
    \iff \mathrm{I} \underset{\mathrm{pre}}{\uparrow} \mathsf G_1(\mathscr N_X[\mathcal A], \neg \Gamma_X(\mathcal B))\]
\end{theorem}
\begin{corollary} \label{cor:WeakerThanMetrizability}
    For any regular space \(X\), the following are equivalent.
    \begin{enumerate}[label=(\roman*)]
        \item
        \(\mathrm{II} \uparrow \mathsf G_1(\Omega_X,\Omega_X)\).
        \item
        \(\mathrm{II} \uparrow \mathsf G_1(\Omega_{MU_p(X),\mathbf 0},\Omega_{MU_p(X),\mathbf 0})\).
        \item
        \(\mathrm{II} \uparrow \mathsf G_1(\mathscr D_{MU_p(X)},\Omega_{MU_p(X),\mathbf 0})\).
        \item
        \(\mathrm{I} \uparrow \mathsf G_1(\mathscr T_{MU_p(X)}, \mathrm{CD}_{MU_p(X)})\).
        \item
        \(\mathrm{I} \uparrow \mathsf G_1(\mathscr N_X[[X]^{<\omega}], \neg \Omega_X)\).
        \item
        \(\mathrm{I} \uparrow \mathsf G_1(\mathscr N_X[[X]^{<\omega}], \neg \Gamma_\omega(X))\).
    \end{enumerate}
    Also, the following are equivalent.
    \begin{enumerate}[label=(\roman*)]
        \item
        \(\mathrm{II} \uparrow \mathsf G_1(\mathcal K_X,\mathcal K_X)\).
        \item
        \(\mathrm{II} \uparrow \mathsf G_1(\Omega_{MU_k(X),\mathbf 0},\Omega_{MU_k(X),\mathbf 0})\).
        \item
        \(\mathrm{II} \uparrow \mathsf G_1(\mathscr D_{MU_k(X)},\Omega_{MU_k(X),\mathbf 0})\).
        \item
        \(\mathrm{I} \uparrow \mathsf G_1(\mathscr T_{MU_k(X)}, \mathrm{CD}_{MU_k(X)})\).
        \item
        \(\mathrm{I} \uparrow \mathsf G_1(\mathscr N_X[K(X)], \neg \mathcal K_X)\).
        \item
        \(\mathrm{I} \uparrow \mathsf G_1(\mathscr N_X[K(X)], \neg \Gamma_k(X)\).
    \end{enumerate}
\end{corollary}
In general, Corollaries \ref{cor:ParticularMetrizability} and \ref{cor:WeakerThanMetrizability} are strictly separate,
as the following example demonstrates.
\begin{example}
    Let \(X\) be the one-point Lindel{\"{o}}fication of \(\omega_1\) with the discrete topology.
    In \cite[Ex. 3.24]{CHCompactOpen}, it is shown that \(X\) has the property that
    \(\mathrm{II} \uparrow \mathsf G_1(\mathcal K_X,\mathcal K_X)\), but
    \(\mathrm{II} \underset{\mathrm{mark}}{\not\uparrow} \mathsf G_1(\mathcal K_X,\mathcal K_X)\).
\end{example}

However, according to Theorem \ref{thm:PawlikowskiStuff}, if Two can win against predetermined strategies
in some Rothberger-like games, Two can actually win against full-information strategies in those games.
\begin{theorem} \label{thm:PawlikowskiStuff}
    Let \(X\) be any space.
    \begin{enumerate}[label=(\roman*)]
        \item
        By Pawlikowski \cite{Pawlikowski},
        \[\mathrm{I} \underset{\mathrm{pre}}{\uparrow} \mathsf G_1(\mathcal O_X , \mathcal O_X) \iff
        \mathrm{I} \uparrow \mathsf G_1(\mathcal O_X , \mathcal O_X).\]
        \item
        By Scheepers \cite{ScheepersIII} (see also \cite[Cor. 4.12]{CHVietoris}),
        \[\mathrm{I} \underset{\mathrm{pre}}{\uparrow} \mathsf G_1(\Omega_X , \Omega_X) \iff
        \mathrm{I} \uparrow \mathsf G_1(\Omega_X , \Omega_X).\]
        \item
        By \cite[Thm. 4.21]{CHVietoris},
        \[\mathrm{I} \underset{\mathrm{pre}}{\uparrow} \mathsf G_1(\mathcal K_X , \mathcal K_X) \iff
        \mathrm{I} \uparrow \mathsf G_1(\mathcal K_X , \mathcal K_X).\]
    \end{enumerate}
\end{theorem}
\begin{corollary} \label{cor:BigRothberger}
    For any regular space \(X\), the following are equivalent.
    \begin{enumerate}[label=(\roman*)]
        \item
        \(X\) is \(\omega\)-Rothberger.
        \item
        \(X^{<\omega}\) is Rothberger, where \(X^{<\omega}\) is the disjoint union of \(X^n\) for all \(n \geq 1\).
        (See \cite{Sakai1988} and \cite[Cor. 3.11]{CHVietoris}.)
        \item
        \(\mathcal{P}_{\mathrm{fin}}(X)\) is Rothberger, where \(\mathcal{P}_{\mathrm{fin}}(X)\)
        is the set \([X]^{<\omega}\) with the subspace topology inherited from \(\mathbb K(X)\).
        (See \cite[Cor. 4.11]{CHVietoris}.)
        \item
        \(\mathrm{I} \not\uparrow \mathsf G_1(\Omega_X , \Omega_X)\).
        \item
        \(MU_p(X)\) has strong countable fan-tightness at \(\mathbf 0\).
        \item
        \(MU_p(X)\) has strong countable dense fan-tightness at \(\mathbf 0\).
        \item
        \(\mathrm{II} \underset{\mathrm{mark}}{\not\uparrow} \mathsf G_1(\mathscr N_X[[X]^{<\omega}], \neg \Omega_X)\).
        \item
        \(\mathrm{II} \not\uparrow \mathsf G_1(\mathscr N_X[[X]^{<\omega}], \neg \Omega_X)\).
        \item
        \(\mathrm{II} \underset{\mathrm{mark}}{\not\uparrow} \mathsf G_1(\mathscr T_{MU_p(X)}, \mathrm{CD}_{MU_p(X)})\).
        \item
        \(\mathrm{II} \not\uparrow \mathsf G_1(\mathscr T_{MU_p(X)}, \mathrm{CD}_{MU_p(X)})\).
        \item
        \(\mathrm{II} \underset{\mathrm{mark}}{\not\uparrow} \mathsf G_1(\mathscr N_{MU_p(X),\mathbf 0}, \neg \Omega_{MU_p(X),\mathbf 0})\).
        \item
        \(\mathrm{II} \not\uparrow \mathsf G_1(\mathscr N_{MU_p(X),\mathbf 0}, \neg \Omega_{MU_p(X),\mathbf 0})\).
    \end{enumerate}
    Also, the following are equivalent.
    \begin{enumerate}[label=(\roman*)]
        \item
        \(X\) is \(k\)-Rothberger.
        \item
        \(\mathrm{I} \not\uparrow \mathsf G_1(\mathcal K_X , \mathcal K_X)\).
        \item
        \(MU_k(X)\) has strong countable fan-tightness at \(\mathbf 0\).
        \item
        \(MU_k(X)\) has strong countable dense fan-tightness at \(\mathbf 0\).
        \item
        \(\mathrm{II} \underset{\mathrm{mark}}{\not\uparrow} \mathsf G_1(\mathscr N_X[K(X)], \neg \mathcal K_X)\).
        \item
        \(\mathrm{II} \not\uparrow \mathsf G_1(\mathscr N_X[K(X)], \neg \mathcal K_X)\).
        \item
        \(\mathrm{II} \underset{\mathrm{mark}}{\not\uparrow} \mathsf G_1(\mathscr T_{MU_k(X)}, \mathrm{CD}_{MU_k(X)})\).
        \item
        \(\mathrm{II} \not\uparrow \mathsf G_1(\mathscr T_{MU_k(X)}, \mathrm{CD}_{MU_k(X)})\).
        \item
        \(\mathrm{II} \underset{\mathrm{mark}}{\not\uparrow} \mathsf G_1(\mathscr N_{MU_k(X),\mathbf 0}, \neg \Omega_{MU_k(X),\mathbf 0})\).
        \item
        \(\mathrm{II} \not\uparrow \mathsf G_1(\mathscr N_{MU_k(X),\mathbf 0}, \neg \Omega_{MU_k(X),\mathbf 0})\).
    \end{enumerate}
\end{corollary}

Using the techniques of Gerlits and Nagy \cite{GerlitsNagy}, Galvin \cite{Galvin1978}, and Telg{\'{a}}rsky \cite{Telgarsky1975},
we offer an analog of Lemma \ref{lem:TelgarskyCofinal} related to the so-called
\emph{weak \(k\)-covering number} in \cite{HolaHoly2022}.
\begin{definition}
    For a collection \(\mathcal A\) of closed subsets of \(X\), define the \textdefn{weak covering number}
    of \(\mathcal A\) to be
    \[wkc(\mathcal A) = \min\left\{ \kappa \in \mathrm{CARD} : (\exists \mathscr F \in [\mathcal A]^\kappa)
    \ X = \mathrm{cl}\left( \bigcup \mathscr F\right) \right\}.\]
\end{definition}
\begin{definition}
    For a space \(X\), let \(\mathrm{DO}_X\) be the set of all \(\mathscr U \subseteq \mathscr T_X\)
    so that \(X = \mathrm{cl}\left( \bigcup \mathscr U \right)\).
\end{definition}
\begin{lemma}
    Let \(\mathcal A \subseteq \wp^+(X)\) where \(X\) is a regular space.
    Then
    \[\mathrm{I} \underset{\mathrm{pre}}{\uparrow} \mathsf G_1(\mathscr N_X[\mathcal A] , \neg \mathrm{DO}_X)
    \iff wkc(\mathcal A) \leq \omega.\]

    If \(X\) is metrizable and \(\mathcal A \subseteq K(X)\), then
    \begin{align*}
        \mathrm{I} \uparrow \mathsf G_1(\mathscr N_X[\mathcal A] , \neg \mathrm{DO}_X)
        &\iff \mathrm{I} \underset{\mathrm{pre}}{\uparrow} \mathsf G_1(\mathscr N_X[\mathcal A] , \neg \mathrm{DO}_X)\\
        &\iff wkc(\mathcal A) \leq \omega.
    \end{align*}
\end{lemma}
\begin{proof}
    The implications
    \[wkc(\mathcal A) \leq \omega
    \implies \mathrm{I} \underset{\mathrm{pre}}{\uparrow} \mathsf G_1(\mathscr N_X[\mathcal A] , \neg \mathrm{DO}_X)\]
    and
    \[\mathrm{I} \underset{\mathrm{pre}}{\uparrow} \mathsf G_1(\mathscr N_X[\mathcal A] , \neg \mathrm{DO}_X)
    \implies \mathrm{I} \uparrow \mathsf G_1(\mathscr N_X[\mathcal A] , \neg \mathrm{DO}_X)\]
    are evident and hold with no assumptions on \(X\).

    Suppose \(X\) is regular and that \(wkc(\mathcal A) > \omega\).
    Any predetermined strategy for One in \(\mathsf G_1(\mathscr N_X[\mathcal A] , \neg \mathrm{DO}_X)\)
    corresponds to a sequence \(\langle A_n : n \in \omega \rangle \in \mathcal A^\omega\).
    We show that no such strategy can be winning for One.
    Since \(\bigcup\{A_n:n\in\omega\}\) is not dense in \(X\), we can find \(U \in \mathscr T_X\)
    so that \(U \cap \bigcup\{A_n:n\in\omega\} = \emptyset\).
    As \(X\) is assumed to be regular, without loss of generality, we can assume that
    \(\mathrm{cl}(U) \cap \bigcup\{A_n:n\in\omega\} = \emptyset\).
    It follows that \(A_n \subseteq X \setminus \mathrm{cl}(U) =: V\)
    for all \(n \in \omega\).
    Hence, \(\langle \mathscr N_X(A_n) : n \in \omega \rangle\) is not a winning strategy for One.
    That is, \(\mathrm{I} \underset{\mathrm{pre}}{\not\uparrow}
    \mathsf G_1(\mathscr N_X[\mathcal A] , \neg \mathrm{DO}_X)\).

    Now suppose that \(X\) is metrizable, \(\mathcal A \subseteq K(X)\), and
    \(\mathrm{I} \uparrow \mathsf G_1(\mathscr N_X[\mathcal A] , \neg \mathrm{DO}_X)\).
    Let \(\sigma\) be a winning strategy for One which is coded by elements of \(\mathcal A\).
    Let \(\mathbb T_0 = \langle \rangle\) and, for \(n \in \omega\),
    \[\mathbb T_{n+1} = \left\{ w^\frown \left\langle \sigma(w) ,
    \mathbb B\left(\sigma(w),2^{-\ell}\right)\right\rangle : w \in \mathbb T_n \wedge \ell \in \omega \right\}.\]
    Notice that \(\mathscr F = \bigcup_{n\in\omega}\{ \sigma(w) : w \in \mathbb T_n\}\) is a countable subset
    of \(\mathcal A\).

    We show that \(X = \mathrm{cl}\left( \bigcup \mathscr F \right)\) by way of
    contradiction.
    Suppose \(X \setminus \mathrm{cl}\left( \bigcup \mathscr F \right) \neq \emptyset\)
    and let \(V \in \mathscr T_X\) be open so that
    \(\mathrm{cl}(V) \cap \mathrm{cl}\left( \bigcup \mathscr F \right) = \emptyset\).
    For \(A_0 := \sigma(\emptyset)\), let \(\ell_0 \in \omega\) be so that
    \(\mathrm{cl}(V) \cap \mathbb B\left(A_0,2^{-\ell_0}\right) = \emptyset\) and set
    \(w_0 = \langle A_0 , \mathbb B\left(A_0,2^{-\ell_0}\right) \rangle \in \mathbb T_1\).

    For \(n \in \omega\), suppose we have \(w_n \in \mathbb T_{n+1}\) defined.
    Now, for \(A_{n+1} := \sigma(w_n)\), let \(\ell_{n+1} \in \omega\) be so that
    \(\mathrm{cl}(V) \cap \mathbb B\left(A_{n+1},2^{-\ell_{n+1}}\right) = \emptyset\).
    Then define \(w_{n+1} = w_n^\frown \left\langle A_{n+1} ,
    \mathbb B\left(A_{n+1},2^{-\ell_{n+1}}\right) \right\rangle \in \mathbb T_{n+2}\).

    Now, we have a run of the game according to \(\sigma\) with the property that
    \[\mathrm{cl}(V) \cap \bigcup_{n\in\omega} \mathbb B\left(A_n, 2^{-\ell_n}\right) = \emptyset.\]
    It follows that
    \[V \cap \mathrm{cl}\left( \bigcup_{n\in\omega} \mathbb B\left(A_n, 2^{-\ell_n}\right) \right) = \emptyset,\]
    which contradicts the assumption that \(\sigma\) is a winning strategy.
\end{proof}
Consequently, for a regular space \(X\),
\begin{itemize}
    \item
    \(\mathrm{I} \underset{\mathrm{pre}}{\uparrow} \mathsf G_1(\mathscr N_X[[X]^{<\omega}] , \neg \mathrm{DO}_X)\)
    if and only if \(X\) is separable and
    \item
    \(\mathrm{I} \underset{\mathrm{pre}}{\uparrow} \mathsf G_1(\mathscr N_X[K(X)] , \neg \mathrm{DO}_X)\)
    if and only if \(X\) admits a countable collection of compact subsets whose union is dense.
\end{itemize}
Also, by Lemma \ref{lem:DualGames}, \(\mathsf G_1(\mathscr N_X[\mathcal A] , \neg \mathrm{DO}_X)\)
is dual to \(\mathsf G_1(\mathcal O_X(\mathcal A), \mathrm{DO}_X)\).
Hence, we obtain
\begin{corollary} \label{cor:FinalCorollary}
    If \(X\) is a regular space, then
    \begin{align*}
        X \text{ is separable}
        &\iff \mathrm{I} \underset{\mathrm{pre}}{\uparrow} \mathsf G_1(\mathscr N_X[[X]^{<\omega}] , \neg \mathrm{DO}_X)\\
        &\iff \mathrm{II} \underset{\mathrm{mark}}{\uparrow} \mathsf G_1(\Omega_X , \mathrm{DO}_X)
    \end{align*}
    and
    \begin{align*}
        wkc(K(X)) \leq \omega
        &\iff \mathrm{I} \underset{\mathrm{pre}}{\uparrow} \mathsf G_1(\mathscr N_X[K(X)] , \neg \mathrm{DO}_X)\\
        &\iff \mathrm{II} \underset{\mathrm{mark}}{\uparrow} \mathsf G_1(\mathcal K_X , \mathrm{DO}_X).
    \end{align*}
    If \(X\) is, in addition, metrizable, then
    \begin{align*}
        X \text{ is separable}
        &\iff \mathrm{I} \underset{\mathrm{pre}}{\uparrow} \mathsf G_1(\mathscr N_X[[X]^{<\omega}] , \neg \mathrm{DO}_X)\\
        &\iff \mathrm{I} \uparrow \mathsf G_1(\mathscr N_X[[X]^{<\omega}] , \neg \mathrm{DO}_X)\\
        &\iff \mathrm{II} \underset{\mathrm{mark}}{\uparrow} \mathsf G_1(\Omega_X , \mathrm{DO}_X)\\
        &\iff \mathrm{II} \uparrow \mathsf G_1(\Omega_X , \mathrm{DO}_X)
    \end{align*}
    and
    \begin{align*}
        wkc(K(X)) \leq \omega
        &\iff \mathrm{I} \underset{\mathrm{pre}}{\uparrow} \mathsf G_1(\mathscr N_X[K(X)] , \neg \mathrm{DO}_X)\\
        &\iff \mathrm{I} \uparrow \mathsf G_1(\mathscr N_X[K(X)] , \neg \mathrm{DO}_X)\\
        &\iff \mathrm{II} \underset{\mathrm{mark}}{\uparrow} \mathsf G_1(\mathcal K_X , \mathrm{DO}_X)\\
        &\iff \mathrm{II} \uparrow \mathsf G_1(\mathcal K_X , \mathrm{DO}_X).
    \end{align*}
\end{corollary}

\section{Questions} \label{sec:Questions}

It is tempting to conjecture that an analogous result to Theorem \ref{thm:Group} holds for the
space of minimal usco maps, except that the natural candidate of pointwise set difference between two minimal usco maps
may not generally produce a minimal usco map.
For \(A,B \subseteq \mathbb R\), let \(A - B = \{ x-y : x \in A, y \in B \}\) and consider \(\Phi := \overline{\one}_{[0,1]}\).
Note that the pointwise difference \(\Phi - \Phi : X \to \mathbb R\) defined by \(\Phi(x) - \Phi(x)\) has the property that
\[\Phi(x) - \Phi(x) = \begin{cases}
    \{0\}, & x \not\in \{ 0,1 \};\\
    \{-1,0,1\}, & x \in \{0,1\}.
\end{cases}\]
Since \(\mathrm{gr}(\mathbf 0) \subseteq \mathrm{gr}(\Phi - \Phi)\), we see that \(\Phi - \Phi\) is not minimal.

We also show that the pointwise Pompeiu-Hausdorff distance between two minimal usco maps
need not be quasicontinuous.
For compact \(A, B \subseteq \mathbb R\), let
\[H_d(A,B) = \inf\{\varepsilon > 0 : A \subseteq \mathbb B(B,\varepsilon) \wedge B \subseteq \mathbb B(A,\varepsilon) \}.\]
We can define \(f : X \to \mathbb R\) by
\(f(x) = H_d\left(\overline{\one}_{[0,1]}(x) , \overline{\one}_{\mathbb R\setminus [0,1]}(x)\right)\)
and observe that
\[f(x) = \begin{cases}
    1, & x \not\in \{0,1\};\\
    0, & x\in\{0,1\}.
\end{cases}\]
So \(f\) is not quasicontinuous.
However, this doesn't mean that the pointwise Pompeiu-Hausdorff distance cannot be used to establish
an analogous result to Theorem \ref{thm:Group} for \(MU_{\mathcal A}(X)\).

We end with a few questions.
\begin{question}
    Is there a result similar to Theorem \ref{thm:Group} for general uniform spaces?
    Along these lines, are there selection games that characterize the diagonal degree
    and the uniformity degree of a uniform space?
\end{question}
\begin{question}
    Can results similar to Theorems \ref{thm:FirstEquivalence} and \ref{thm:SecondTheorem} be established
    relative to \(\Omega_{MU_{\mathcal A}(X),\Phi}\) for any \(\Phi \in MU(X)\)?
\end{question}
 \begin{question}
    How many of the equivalences and dualities of this paper can be established for games of longer length
    and for finite-selection games?
 \end{question}
\begin{question}
   How much of this theory can be recovered when we study \(MU(X,Y)\) for \(Y \neq \mathbb R\), for example,
    when \(Y\) is \([0,1]\), any metrizable space, any topological group, or any uniform space?
\end{question}

\providecommand{\bysame}{\leavevmode\hbox to3em{\hrulefill}\thinspace}
\providecommand{\MR}{\relax\ifhmode\unskip\space\fi MR }
\providecommand{\MRhref}[2]{%
  \href{http://www.ams.org/mathscinet-getitem?mr=#1}{#2}
}
\providecommand{\href}[2]{#2}

\end{document}